\documentclass[review]{siamart1116}

\usepackage{bm}
\usepackage{lipsum}
\usepackage{amsfonts}
\usepackage{graphicx}
\usepackage{epstopdf}
\usepackage{enumitem}
\ifpdf
  \DeclareGraphicsExtensions{.eps,.pdf,.png,.jpg}
\else
  \DeclareGraphicsExtensions{.eps}
\fi

\numberwithin{theorem}{section}

\newcommand{\TheTitle}{Analysis of Optimization Algorithms via Integral Quadratic Constraints: Nonstrongly Convex Problems} 
\newcommand{\TheAuthors}{Mahyar Fazlyab, Alejandro Ribeiro, Manfred Morari, and Victor M. Preciado}

\headers{A Unified Analysis of Optimization Algorithms}{M. Fazlyab, A. Ribeiro, M. Morari, and V.M. Preciado}

\title{{\TheTitle}\thanks{Submitted to the editors DATE.
\funding{This work was supported in part by the NSF under
grants CAREER-ECCS-1651433 and IIS-1447470.}}}

\author{
  Mahyar Fazlyab\thanks{Department of Electrical and Systems Engineering, University of Pennsylvania, Philadelphia, PA
    \email{mahyarfa@seas.upenn.edu}.}
  \and
  Alejandro Ribeiro\footnotemark[2] 
  \and
  Manfred Morari\footnotemark[2]
  \and
  Victor~M.~Preciado\footnotemark[2]
}

\usepackage{amsopn}

\usepackage{notoccite}
\usepackage{graphics} 
\usepackage{amsmath} 
\usepackage{amssymb}  
\usepackage{theorem}
\usepackage{color,xspace}
\usepackage{graphicx}
\usepackage{subcaption}
\usepackage{cite}
\usepackage{cleveref}
\usepackage{bm}
\usepackage{bbm}
\usepackage{tikz}
\usepackage{epstopdf} 
\usepackage{ntheorem}
\usepackage{enumitem}

\def\dom{\mathrm{dom}}

\renewcommand{\textfloatsep}{5 pt}
\renewcommand{\intextsep}{5 pt}

\numberwithin{equation}{section}

\ifpdf
\hypersetup{
  pdftitle={\TheTitle},
  pdfauthor={\TheAuthors}
}
\fi

\newtheorem{remark}{Remark}

\begin{document}

\maketitle

\begin{abstract}
	In this paper, we develop a unified framework able to certify both exponential and subexponential convergence rates for a wide range of iterative first-order optimization algorithms. To this end, we construct a family of parameter-dependent nonquadratic Lyapunov functions that can generate convergence rates in addition to proving asymptotic convergence. Using Integral Quadratic Constraints (IQCs) from robust control theory, we propose a Linear Matrix Inequality (LMI) to guide the search for the parameters of the Lyapunov function in order to establish a rate bound. Based on this result, we formulate a Semidefinite Programming (SDP) whose solution yields the best convergence rate that can be certified by the class of Lyapunov functions under consideration. We illustrate the utility of our results by analyzing the gradient method, proximal algorithms and their accelerated variants for (strongly) convex problems.  We also develop the continuous-time counterpart, whereby we analyze the gradient flow and the continuous-time limit of Nesterov's accelerated method. 
\end{abstract}



\section{Introduction}
The analysis and design of iterative optimization algorithms is a well-established research area in optimization theory. Due to their computational efficiency and global convergence properties, first-order methods are of particular interest, especially in large-scale optimization arising in current machine learning applications. However, these algorithms can be very slow, even for moderately well-conditioned problems. In this direction, accelerated variants of first-order algorithms, such as Polyak's Heavy-ball algorithm \cite{POLYAK19641} or Nesterov's accelerated method \cite{nesterov1983method}, have been developed to speed up the convergence in ill-conditioned and nonstrongly convex problems. 

In numerical optimization, convergence analysis is an integral part of algorithm tuning and design. This task, however, is often pursued on a case-by-case basis and the analysis techniques heavily depend on the particular algorithm under study, as well as the underlying assumptions. However, by interpreting iterative algorithms as feedback dynamical systems, it is possible to incorporate tools from control theory to analyze and design these algorithms in a more systematic and unified manner \cite{hu2017control,wang2011control,feijer2010stability,wang2010controlapproach}. Moreover, control techniques can be exploited to address more complex tasks, such as analyzing robustness against uncertainties, deriving nonconservative worst-case bounds, and providing convergence guarantees under less restrictive assumptions \cite{cherukuri2016role,hu2017control,lessard2016analysis}.


A universal approach to analyzing the stability of dynamical systems is to construct a Lyapunov function that decreases along the trajectories of the system, proving asymptotic convergence. In the context of iterative optimization algorithms, it is of particular importance to certify a convergence rate in addition to proving asymptotic convergence. Construction of Lyapunov functions that can achieve this goal is not straightforward, especially for nonstrongly convex problems, in which the convergence rate is subexponential. It is important to remark that in a considerable number of applications in machine learning, the underlying optimization problem is not strongly convex \cite{beck2009fast}. 


The goal of the present work is to develop a semidefinite programming (SDP) framework for the construction of Lyapunov functions that can characterize both exponential and subexponential convergence rates for iterative first-order optimization algorithms. The main pillars of our framework are time-varying Lyapunov functions, originally proposed in \cite{su2016differential} for analyzing gradient-based momentum methods \cite{wibisono2016variational,wilson2016lyapunov}, as well as Integral Quadratic Constraints (IQCs) from robust control theory \cite{yakubovich1967frequency,megretski1997system}, which have recently been adapted by Lessard et al. \cite{lessard2016analysis} in the context of optimization algorithms.
Specifically, we propose a family of nonquadratic Lyapunov functions equipped with time-dependent parameters that can establish both exponential and subexponential convergence rates. We then develop an LMI to guide the search for the parameters of the Lyapunov function in order to generate analytical/numerical convergence rates. Based on this result, we formulate an SDP to compute the fastest convergence rate that can be certified by the class of Lyapunov functions under consideration. In this SDP, the properties of the objective function (e.g., convexity, Lipschitz continuity, etc.) can be systematically encoded into the SDP, providing a modular approach to obtaining convergence rates under various regularity assumptions, such as quasiconvexity \cite{hazan2015beyond}, weak quasiconvexity \cite{hardt2016gradient}, quasi-strong convexity \cite{necoara2015linear}, quadratic growth \cite{necoara2015linear}, and Polyak-$\L{}$ojasiewicz condition \cite{karimi2016linear}. Furthermore, we extend our framework to continuous-time settings, in which we analyze the continuous-time limits (by taking infinitesimal stepsizes) of relevant iterative optimization algorithms. We will illustrate the generality of our framework by analyzing several first-order optimization algorithms; namely, unconstrained (accelerated) gradient methods, gradient methods with projection, and (accelerated) proximal methods.

Finally, we consider algorithm design. Specifically, we develop a robust counterpart of the developed LMI whose feasibility provides the algorithm with an additional stability margin in the sense of Lyapunov. As a design experiment, we use the LMI to tune the stepsize and momentum coefficient of Nesterov's accelerated method applied to strongly convex functions, considering robustness as a design criterion.

%
%
%
%

\subsection{Related work} 
There is a host of results in the literature using SDPs to analyze the convergence of first-order optimization algorithms \cite{drori2014performance,taylor2017smooth,taylor2017exact,kim2016optimized}. The first among them is \cite{drori2014performance}, in which Drori and Teboulle developed an SDP to derive analytical/numerical bounds on the worst-case performance of the unconstrained gradient method and its accelerated variant. An extension of this framework to the proximal gradient method--for the case of strongly convex problems--has been recently proposed in \cite{taylor2017exact}. These SDP formulations, despite
being able to yield new performance bounds, are highly algorithm dependent.
To depart from classical algorithmic view, Lessard et. al \cite{lessard2016analysis}  developed an SDP framework based on quadratic Lyapunov functions and IQCs to derive sufficient conditions for exponential stability of an algorithm when the objective function is strongly convex\cite[Theorem 4]{lessard2016analysis}. Specifically, they formulate a small SDP whose feasibility verifies exponential convergence at a specified rate. It is important to remark that Lessard's framework is specifically tailored to analyze strongly convex problems with exponential convergence \cite{lessard2016analysis,nishihara2015general} and subexponential rates cannot be captured. Finally, another related work is by Hu and Lessard \cite{hu2017nesterov}, in which they have independently proposed an LMI framework based on quadratic Lyapunov functions and dissipativity theory to analyze Nesterov's accelerated method. In contrast, the present work, inspired by \cite{lessard2016analysis}, develops an IQC framework using time-dependent nonquadratic Lyapunov functions for the analysis of a broader family of functionals, as well as algorithms involving projections and proximal operators, including the proximal variant of Nesterov's method. 
%


\subsection{Notation and preliminaries}
We denote the set of real numbers by $\mathbb{R}$, the set of real $n$-dimensional vectors by $\mathbb{R}^n$, the set of $m\times n$-dimensional matrices by $\mathbb{R}^{m\times n}$, and the $n$-dimensional identity matrix by $I_n$. We denote by $\mathbb{S}^{n}$, $\mathbb{S}_{+}^n$, and $\mathbb{S}_{++}^n$ the sets of $n$-by-$n$ symmetric, positive semidefinite, and positive definite matrices, respectively. For $M \in \mathbb{R}^{n \times n}$ and $x \in \mathbb{R}^n$, we have that $x^\top M x = \frac{1}{2} x^\top (M+M^\top) x$. The $p$-norm ($p \geq 1$) is displayed by $\|\cdot\|_p \colon \mathbb{R}^n \to \mathbb{R}_{+}$. For two matrices $A \in \mathbb{R}^{m\times n}$ and $B\in \mathbb{R}^{p\times q}$ of arbitrary dimensions, their Kronecker product is given by 
$$
A \otimes B = \begin{bmatrix}
A_{11} B & \cdots & A_{1n}B \\
\vdots & \ddots & \vdots \\ A_{m1}B & \ldots & A_{mn} B
\end{bmatrix}.
$$
Further, we have that $(A\otimes B)^\top =A^\top \otimes B^\top$ and $(AC) \otimes (BD)=(A\otimes B)(C \otimes D)$, for matrices of appropriate dimensions. 
Let $f \colon \mathbb{R}^n \to \mathbb{R} \cup \{+\infty\}$ be a closed proper function. The effective domain of $f$ is denoted by $\dom \, f=\{x \in \mathbb{R}^n \colon f(x)<\infty\}$. The indicator function $\mathbb{I}_{\mathcal{X}} \colon \mathbb{R}^n \to \mathbb{R} \cup \{+\infty\}$ of a closed nonempty convex set $\mathcal{X} \subset \mathbb{R}^n$ is defined as $\mathbb{I}_{\mathcal{X}}(x)=0$ if $x \in \mathcal{X}$, and $\mathbb{I}_{\mathcal{X}}(x)=+\infty$ otherwise. The Euclidean projection of $x \in \mathbb{R}^n$ onto a set $\mathcal{X}$ is denoted by $[x]_{\mathcal{X}}=\mathrm{argmin}_{y\in \mathcal{X}} \|y-x\|_2$. 

\begin{definition}[Smoothness]
	A differentiable function $f \colon \mathbb{R}^d \to \mathbb{R}$ is $L_f$-smooth on $\mathcal{S} \subseteq \dom \, f$ if the following inequality holds.
	\begin{align} \label{eq: Lipschitz continuity}
	\|\nabla f(x)-\nabla f(y)\|_2 \leq L_f \|x-y\|_2 \quad \mbox{for all} \ x,y \in \mathcal{S}.
	\end{align}
	An equivalent definition is that
	\begin{align} \label{eq: Lipschitz continuity function values}
	f(y) \leq f(x) + \nabla f(x)^\top (y-x) + \dfrac{L_f}{2} \|y-x\|_2^2 \quad  \text{for all} \ x,y \in \mathcal{S}.
	\end{align}
\end{definition}

\begin{definition}[Strong convexity]
	A differentiable function $f \colon \mathbb{R}^d \to \mathbb{R}$ is $m_f$-strongly convex on $\mathcal{S} \subseteq \dom \, f$ if the following inequality holds.
	\begin{align} \label{eq: strong convexity}
	m_f \|x-y\|_2^2 \leq (x-y)^\top (\nabla f(x)-\nabla f(y)) \quad  \text{for all} \ x,y \in \mathcal{S}.
	\end{align}
	An equivalent definition is that
	\begin{align} \label{eq: strong convexity function values}
	f(x)+\nabla f(x)^\top (y-x) + \dfrac{m_f}{2} \|y-x\|_2^2  \leq f(y)   \quad  \text{for all}  \ x,y \in \mathcal{S}.
	\end{align}
	
\end{definition}
We denote the class of $L_f$-smooth and $m_f$-strongly convex functions by $\mathcal{F}(m_f,L_f)$. Note that, by setting $m_f=0$, we recover convex functions. For the class $\mathcal{F}(m_f,L_f)$, we denote the condition number by $\kappa_f=L_f/m_f \geq 1$.

\section{Algorithm representation} \label{sec: Algorithm representation}
Iterative algorithms can be represented as linear dynamical systems interacting with one or more static nonlinearities \cite{lessard2016analysis}. The linear part describes the algorithm itself, while the nonlinear components depend exclusively on the first-order oracle of the objective function. In this paper, we consider first-order algorithms that have the following state-space representation,
\begin{align} \label{eq: discrete time dynamics}
\xi_{k+1} &= A_k \xi_{k}  + B_k u_k, \\ \nonumber y_k &= C_k \xi_k, \\ \nonumber u_k &=\phi(y_k), \\ \nonumber x_k &= E_k \xi_k, 
\end{align}
where at each iteration index $k$, $\xi_{k} \in \mathbb{R}^n$ is the state, $u_k \in \mathbb{R}^d$ is the input ($d \leq n$), $y_k \in \mathbb{R}^d$ is the feedback output that is transformed by the nonlinear map $\phi \colon \mathbb{R}^d \to \mathbb{R}^d$ to generate $u_k$, and $x_k \in \mathbb{R}^d$ is the output at which the suboptimality will be evaluated for convergence analysis. See Figure \ref{fig: Fig_block_diagram_discrete} for a block diagram representation.\footnote{Since the input $u =\phi(y)$ is an explicit function of the output, we set the feedforward matrix $D$ to zero in the representation of the linear dynamics to ensure the explicit dependence of the feedback input on the output, i.e., the feedback system is well-posed.}

\begin{figure}[htbp]
	\centering
	\includegraphics[width=0.35\textwidth]{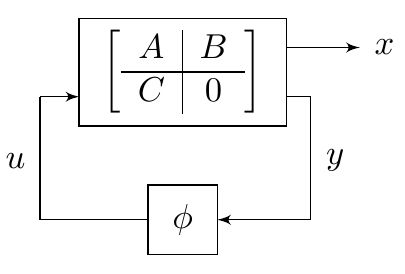}
	\caption{\small Block diagram representation of a first-order algorithm in state-space form.}
	\label{fig: Fig_block_diagram_discrete}
\end{figure}

A broad family of first-order algorithms can be represented in the canonical form \eqref{eq: discrete time dynamics}, where the matrices $(A_k, B_k, C_k,E_k)$ differ for each algorithm. In this representation, the nonlinear feedback component $\phi$ depends on the oracle of the objective function. For instance, in unconstrained smooth minimization problems, we have that $\phi=\nabla f$, where $f$ is the objective function. In composite optimization problems, $\phi$ is the generalized gradient mapping of the composite function, which we will describe in $\S$\ref{sec: Proximal methods}. As an illustration, consider the following recursion defined on the two sequences $\{x_k\} $ and $\{y_k\} $,
\begin{align} \label{eq: universal example}
x_{k+1} &= x_k+\beta_k(x_k-x_{k-1})-h_k \nabla f(y_k), \\
y_k &= x_k+\gamma_k(x_k-x_{k-1}), \nonumber 
\end{align}
where $h_k, \beta_k$ and $\gamma_k$ are nonnegative scalars, $\{x_k\} $ is the primary sequence, and $\{y_k\} $ is the sequence at which the gradient is evaluated. By defining the state vector $\xi_k = [x_{k-1}^\top \ x_k^\top]^\top \in \mathbb{R}^{2d}$, we can represent \eqref{eq: universal example} in the canonical form \eqref{eq: discrete time dynamics}, 
where the matrices $(A_k,B_k,C_k)$ are given by
	\begin{align}
	\left[
	\begin{array}{c|c}
	A_k & B_k \\
	\hline
	C_k & 0
	\end{array}
	\right] = \left[
	\begin{array}{c|c}
	\begin{matrix}
	0 &I_d \\ -\beta_kI_d & (\beta_k+1)I_d
	\end{matrix} & \begin{matrix} 0 \\ -h_kI_d\end{matrix} \\
	\hline
	\begin{matrix} -\gamma_kI_d & (\gamma_k+1)I_d\end{matrix} & 0
	\end{array}
	\right].
	\end{align}
Notice that depending on the selection of $\beta_k$ and $\gamma_k$, \eqref{eq: universal example} describes various existing algorithms. For example, the gradient method corresponds to the case $\beta_k=\gamma_k=0$. In Nesterov's accelerated method, we have $\beta_k=\gamma_k$. Finally, we recover the Heavy-ball method by setting $\gamma_k=0$. 

For an algorithm represented in the canonical form \eqref{eq: discrete time dynamics}, its fixed points (if they exist) are characterized by
\begin{align} \label{eq: fixed points}
\xi_{\star} = A_k \xi_{\star} + B_k u_{\star}, \ \  y_\star = C_k \xi_{\star}, \ \  u_{\star} =\phi(y_\star), \ \ x_{\star}=E_k \xi_{\star} \ \ \mbox{for all } k.
\end{align}
For well-designed algorithms, the fixed-point equation must coincide with the optimality conditions of the underlying optimization problem.


\section{Main results} \label{section: Discrete-Time Gradient Models} 
In this paper, we are concerned with the convergence analysis of first-order algorithms designed to solve optimization problems of the form
\begin{align} \label{eq: main optimization problem}
\mathcal{X}_{\star} =\mathrm{argmin}_{x \in \mathbb{R}^d} \{F(x) = f(x) + g(x)\},
\end{align}
where $f \colon \mathbb{R}^d \to \mathbb{R}$ is closed, proper, and differentiable, while $g \colon \mathbb{R}^d \to \mathbb{R} \cup \{+\infty\}$ is closed convex proper (CCP), and possibly nondifferentiable. Depending on the choice of $f$ and $g$, \eqref{eq: main optimization problem} describes various specialized optimization problems. For instance, when $g(x)=\mathbb{I}_{\mathcal{X}}(x)$ is the indicator function of a nonempty, closed, convex set $\mathcal{X} \subseteq \mathbb{R}^d$, \eqref{eq: main optimization problem} is equivalent to constrained smooth programming; when $g(x) \equiv 0$, we obtain unconstrained smooth programming; and, when $f(x)\equiv 0$, \eqref{eq: main optimization problem} simplifies to an unconstrained nonsmooth optimization problem. In all cases, we assume that the optimal solution set $\mathcal{X}_{\star}$ is nonempty and closed, and the optimal value $F_{\star} = \inf_{x \in \mathbb{R}^d} \ F(x)$ is finite. 

Consider an iterative first-order algorithm, represented in the state-space form \eqref{eq: discrete time dynamics}, that under appropriate initialization solves \eqref{eq: main optimization problem} asymptotically; that is, the sequence of outputs $\{x_k\}$ satisfies $\lim_{k \to \infty} F(x_k)=F(x_{\star})$, where $x_{\star} \in \mathcal{X}_{\star}$.
We assume that the fixed point $y_{\star}$ of the sequence $\{y_k\}$, defined in \eqref{eq: fixed points}, satisfies $y_{\star}=x_{\star}$.	In other words, both $\{x_k\}$ and $\{y_k\}$ are convergent to the same optimal point $x_{\star}$. To establish a rate bound for the algorithm under study, we propose the following Lyapunov function:
%
\begin{align} \label{eq: Lyapunov gradient discrete}
V_k(x,\xi) = a_k(F(x)-F(x_{\star})) + (\xi-\xi_\star)^\top P_k  (\xi-\xi_\star), 
\end{align}
where $a_k \geq 0, \ P_k \in \mathbb{S}_+^n \ \mbox{for all} \ k$, and are to be determined. The first term is the suboptimality of $x$ scaled by $a_k$ and the second term quantifies the suboptimality of the state $\xi$ with respect to the optimal state $\xi_{\star}$. Notice that by this definition, we have that $V_k(x,\xi) \geq 0$ for all $k$, and $V_k(x_{\star},\xi_{\star})=0$, i.e., the Lyapunov function is nonnegative everywhere and zero at optimality.
Suppose we select $\{a_k\}$ and $\{P_k\}$ such that the Lyapunov function becomes nonincreasing along the trajectories of \eqref{eq: discrete time dynamics}, i.e., the following condition holds.
\begin{align} \label{eq: Lyapunov drift}
V_{k+1}(x_{k+1},\xi_{k+1}) \leq V_k(x_k,\xi_k)\quad \text{for all} \ k.
\end{align}
Then, we can conclude $a_k (F(x_k)-F(x_{\star})) \leq V_k(x_k,\xi_k) \leq V_0(x_0,\xi_0)$, or equivalently,
\begin{align} \label{eq: convergence rate discrete}
0 \leq F(x_k)-F(x_{\star}) \leq \dfrac{V_0(x_0,\xi_0)}{a_k} = \mathcal{O}(\frac{1}{a_k}) \quad \text{for all} \ k.
\end{align}
In other words, the sequence $\{a_k\}$ generates an upper bound on the suboptimality or, equivalently, a lower bound on the convergence rate. As a result, the task of certifying a convergence rate for the algorithm translates into finding sufficient conditions to guarantee \eqref{eq: Lyapunov drift}. In the following theorem, we develop an LMI whose feasibility is sufficient for \eqref{eq: Lyapunov drift} to hold.
%
%
%
\begin{theorem}[Main result] \label{thm: main result}
	Let $x_{\star} \in \mathrm{argmin}_{x \in \mathbb{R}^d} \ F(x)$ be a minimizer of $F \colon \mathbb{R}^d \to \mathbb{R} \cup \{+\infty\}$ with a finite optimal value $F(x_{\star})$. Consider an iterative first-order algorithm in the state-space form \eqref{eq: discrete time dynamics}.
	\medskip
	\begin{enumerate}[leftmargin=*]
		\item Suppose the fixed points $(\xi_{\star},u_{\star},y_{\star},x_{\star})$ of  \eqref{eq: discrete time dynamics} satisfy 
		\begin{align} \label{thm: main result 1}
		\xi_{\star} = A_k \xi_{\star} + B_k u_{\star}, \ \  y_\star = C_k \xi_{\star}, \ \  u_{\star} =\phi(y_\star), \ \ x_{\star}=E_k \xi_{\star}=y_{\star} \ \ \mbox{for all } k.
		\end{align}
		\item Suppose there exist symmetric matrices $M_k^{1}, M_k^{2}, M_k^{3}$ such that the following inequalities hold for all $k$.
		\begin{subequations} \label{thm: main result 2}
			\begin{align}
			F(x_{k+1}) - F(x_k) &\leq e_k^\top M_k^{1} e_k, \label{thm: main result 21}\\
			F(x_{k+1}) - F(x_{\star}) &\leq e_k^\top M_k^{2}e_k, \label{thm: main result 22}\\
			0 &\leq e_k^\top M_k^{3} e_k, \label{thm: main result 23}
			\end{align}
		\end{subequations}
		where $e_k = [(\xi_k-\xi_{\star})^\top \ (u_k-u_{\star})^\top]^\top \in \mathbb{R}^{n+d}$ and $M_k^{3}$ is either zero or indefinite.
		\item Suppose there exists a nonnegative and nondecreasing sequence of reals $\{a_k\}$, a sequence of nonnegative reals $\{\sigma_k\}$, and a sequence of $n \times n$ positive semidefinite matrices $\{P_k\}$ satisfying
		
		\begin{align} \label{thm: main 1}
		M_k^{0}  + a_{k} M_k^{1} + (a_{k+1}-a_k) M_k^{2} + \sigma_k M_k^{3}\preceq 0 \quad \mbox{for all } k,
		\end{align}
		where
		\begin{align} \label{thm: main 2}
		M_k^0 &= \begin{bmatrix}
		A_k^\top P_{k+1} A_k - P_k& A_k^\top P_{k+1} B_k\\  B_k^\top P_{k+1} A_k& B_k^\top P_{k+1} B_k \end{bmatrix}.
		\end{align}
		%
		Then the sequence $\{x_k\}$ satisfies
		\begin{align}
		F(x_{k})-F(x_{\star}) \leq \dfrac{a_0(F(x_0)-F(x_{\star})) + (\xi_0-\xi_\star)^\top P_0  (\xi_0-\xi_\star)}{a_k} \quad \mbox{for all} \ k.
		\end{align}
	\end{enumerate}
\end{theorem}
Before proving Theorem \ref{thm: main result}, we briefly discuss the assumptions made in the statement of the theorem. The first inequality in \eqref{thm: main result 2} bounds the difference between two consecutive iterates. In particular, if $M_k^1$ is negative semidefinite for all $k$, then the sequence $\{F(x_k)\}$ is monotone. The second inequality in \eqref{thm: main result 2} bounds the suboptimality; and finally, the third inequality in \eqref{thm: main result 2} is a quadratic constraint on the input-output pairs $(\xi_k,u_k)$ that are related via the rule $u_k = \phi(C_k \xi_{k})$. These bounds will be required to satisfy condition \eqref{eq: Lyapunov drift} and will feature heavily throughout the paper.
Note that the matrices $(M_k^{1}, M_k^{2},M_k^{3})$ in \eqref{thm: main result 2} depend on the algorithm parameters, i.e., the matrices $(A_k,B_k,C_k,E_k)$ that define the algorithm, as well as the assumptions about the objective function $F$. 

\begin{proof}[Proof of Theorem \ref{thm: main result}]
		First, by \eqref{eq: discrete time dynamics} and \eqref{thm: main result 1}, we can write
		\begin{align*}
		\xi_{k+1}-\xi_{\star}&=A_k(\xi_{k}-\xi_{\star})+B_k(u_k-u_{\star}),
		\end{align*}
		Using the above identity, we can write
		\begin{subequations} \label{thm: main 56}
			\begin{align} \label{thm: main 6}
			(\xi_{k+1}-\xi_{\star})^\top P_{k+1} (\xi_{k+1}-\xi_{\star}) \!-\! (\xi_{k}-\xi_{\star})^\top P_k (\xi_{k}-\xi_{\star}) = e_k^\top M_k^{0} e_k. 
			\end{align}
			Multiply \eqref{thm: main result 21} by $a_{k}$ and \eqref{thm: main result 22} by $(a_{k+1}-a_k)$ and add both sides of the resulting inequalities to obtain
			\begin{align} \label{thm: main 7}
			a_{k+1} (F(x_{k+1})-F(x_{\star})) - a_k (F(x_k)-F(x_{\star})) \leq 0.
			\end{align}
		\end{subequations}
		By adding both sides of the inequalities in \eqref{thm: main 56} and recalling the definition of $V_k(x_k,\xi_k)$ in \eqref{eq: Lyapunov gradient discrete}, we can write
		\begin{align}  \label{thm: main 8}
		V_{k+1}(x_{k+1},\xi_{k+1})\!-\!V_k(x_k,\xi_k) \leq & \, e_k^\top \left( M_k^{0}+a_{k}M_k^{1}+(a_{k+1}-a_k)M_k^{2} \right) e_k.
		\end{align}
		Suppose the matrix inequality in \eqref{thm: main 1} holds. By multiplying this inequality from the left and right by $e_k^\top$ and $e_k$, respectively, we obtain
		\begin{align}  \label{thm: main 9}
		e_k^\top \left( M_k^{0}+a_{k}M_k^{1}+(a_{k+1}-a_k)M_k^{2} + \sigma_k M_k^{3} \right) e_k \leq 0.
		\end{align}
		Finally, adding both sides of \eqref{thm: main 8} and \eqref{thm: main 9} yields
		\begin{align}
		V_{k+1}(x_{k+1},\xi_{k+1}) - V_k(x_k,\xi_k) \leq -\sigma_k  e_k^\top M_k^{3} e_k \leq 0,
		\end{align}
		where the second inequality follows from \eqref{thm: main result 23}. Hence, the sequence $\{V_k(x_k,\xi_k)\}$ is nonincreasing, implying $a_k(F(x_k)-F(x_{\star})) \leq V_k(x_k,\xi_k) \leq V_0(x_0,\xi_0)$. The proof becomes complete by dividing both sides of the last inequality by $a_{k}$.
\end{proof}

Some remarks are in order regarding Theorem \ref{thm: main result}:

\medskip

\begin{enumerate}[leftmargin=*]	
	\item We do not make the assumption that the algorithm under consideration is a descent method. In other words, the sequence $\{F(x_k)\}$ of function values is not necessarily monotone, which is a hallmark of accelerated algorithms \cite{nesterov2013introductory}. In contrast, we require the sequence $\{V_k(x_k,\xi_{k})\} $ of ``energy" values to be monotonically decreasing. From this perspective, the LMI \eqref{thm: main 1} provides a guideline for the construction energy functions with this property.
	\medskip
	\item There is no restriction on the sequence $\{a_k\}$ other than nonnegativity and monotonicity. Hence, we can characterize both exponential ($a_k=\rho^{-k}, 0 \leq \rho<1$) and subexponential ($a_k=k^p, \ p>0$, for example) convergence rates.
	\medskip
	\item We have made no explicit assumptions about the objective function in Theorem \ref{thm: main result}, other than the quadratic bounds in \eqref{thm: main result 2}. In fact, the matrices $M_k^1,M_k^2,M_k^3$ that characterize these bounds depend on the parameters of the algorithm (e.g. stepsize, momentum coefficient, etc.), as well as the assumptions about $F$. In $\S$\ref{section: unconstrained smooth programming} and $\S$\ref{sec: Proximal methods}, we will describe a general procedure for deriving these matrices for a wide range of algorithms and assumptions. 
\end{enumerate}

\medskip

\subsection{Time-invariant algorithms with exponential convergence}
In this subsection, we specialize the results of Theorem \ref{thm: main result} to time-invariant algorithms with exponential convergence. Under these assumptions, we can precondition $a_k$ and $P_k$ to simplify the LMI in \eqref{thm: main 1}. Explicitly, suppose the matrices $(A_k,B_k,C_k,E_k)$ that define the algorithm do not change with $k$. 
By the particular selection 
\begin{align} \label{eq: precondition for SC}
a_k = \rho^{-2k}a_0, \quad a_0>0, \quad P_k = \rho^{-2k} P_0, \quad P_0 \succeq 0, \quad 0 < \rho \leq 1 \quad \mbox{for all } k,
\end{align}
the Lyapunov function in \eqref{eq: Lyapunov gradient discrete} reads as
\begin{align} \label{eq: Lyapunov SC}
V_k(\xi) = \rho^{-2k} \left(a_0(F(x)\!-\!F(x_{\star})) + (\xi-\xi_\star)^\top P_0  (\xi-\xi_\star)\right).
\end{align}
The unknown parameters of the Lyapunov function are now $a_0 >0, P_0 \succeq 0$, and the decay rate $0 < \rho \leq 1$. With this parameter selection, the LMI in \eqref{thm: main 1} simplifies greatly. The following result is a special case of Theorem \ref{thm: main result} for the selection \eqref{eq: precondition for SC}.
\begin{theorem}[Exponential convergence of time-invariant algorithms] \label{thm: exponential rate} In theorem \ref{thm: main result}, assume that the algorithm parameters as well as the matrices $M_k^1,M_k^2,M_k^3$ in \eqref{thm: main result 2} do not change with $k$. In other words,
\begin{align*}
(A_k,B_k,C_k,E_k,M_k^1,M_k^2,M_k^3)=(A_0,B_0,C_0,E_0,M_0^1,M_0^2,M_0^3) \quad \mbox{for all } k.
\end{align*} 
Suppose there exists $a_0>0$, $P_0 \in \mathbb{S}_{+}^n$, and $\lambda_0 \geq 0$ that satisfy
	\begin{align} \label{eq: LMI exponential}
	\begin{bmatrix}
	A_0^\top P_0 A_0 - \rho^2 P_0& A_0^\top P_0 B_0\\  B_0^\top P_0 A_0& B_0^\top P_0 B_0
	\end{bmatrix}+a_0 \rho^2 M_0^1 + a_0(1\!-\!\rho^2) M_0^2 + \lambda_0 M_0^3 \preceq 0,
	\end{align}
	for some $0 < \rho \leq 1$. Then the sequence $\{x_k\}$ satisfy
	\begin{align} \label{eq: exponential convergence 1}
	F(x_k) - F(x_{\star}) \leq \frac{a_0(F(x_0)\!-\!F(x_{\star})) + (\xi_0-\xi_\star)^\top P_0  (\xi_0-\xi_\star)}{a_0}\rho^{2k}.
	\end{align}
\end{theorem}	
\begin{proof}
	By substituting the parameter selection \eqref{eq: precondition for SC} in \eqref{thm: main 1} and factoring out the positive term $\rho^{-2k-2}$ from the resulting LMI, we obtain \eqref{eq: LMI exponential}, which no longer depends on $k$. Utilizing Theorem \ref{thm: main result}, the feasibility of \eqref{eq: LMI exponential} ensures \eqref{eq: Lyapunov drift}, which in turn implies \eqref{eq: exponential convergence 1}. The proof is complete.
\end{proof}
%


\begin{remark}\normalfont
	Regarding the parameter selection in \eqref{eq: precondition for SC}, if we instead select $a_k \equiv 0$, $P_{k}=\rho^{-2k} P_0$ with $P_0 \succ 0$ and $0 < \rho \leq 1$, the Lyapunov function \eqref{eq: Lyapunov gradient discrete} simplifies to the quadratic function
	\begin{align} \label{eq: quadratic Lyapunov}
	V_{k}(\xi)=\rho^{-2k}(\xi-\xi_{\star})^\top P_0 (\xi-\xi_{\star}), \quad P_0 \succ 0.
	\end{align}
	Correspondingly, the LMI \eqref{eq: LMI exponential} in Theorem \ref{thm: exponential rate} reduces to
	\begin{align} \label{eq: Lessard formulation}
	\begin{bmatrix}
	A_0^\top P_0 A_0 - \rho^2 P_0& A_0^\top P_0 B_0\\  B_0^\top P_0 A_0& B_0^\top P_0 B_0
	\end{bmatrix}+\lambda_0  M_0^3 \preceq 0.
	\end{align}
	By Theorem \ref{thm: main result}, if \eqref{eq: Lessard formulation} is feasible for some $P_0 \succ 0, \lambda_0 \geq 0$ and $0 < \rho \leq 1$, then the Lyapunov function in \eqref{eq: quadratic Lyapunov} would satisfy the decreasing property $V_{k+1}(\xi_{k+1}) \leq V_{k}(\xi_{k})$, which translates to
	\begin{align*}
	(\xi_{k+1}-\xi_{\star})^\top P_0 (\xi_{k+1}-\xi_{\star}) \leq \rho^2 (\xi_{k}-\xi_{\star})^\top P_0 (\xi_{k}-\xi_{\star}),
	\end{align*}
	or equivalently,
	\begin{align} \label{eq: exponential convergence 0}
	\|\xi_{k}-\xi_{\star}\|_2^2 \leq \rho^{2k}\text{cond}(P_0)\|\xi_{0}-\xi_{\star}\|_2^2.
	\end{align}
	The matrix inequality \eqref{eq: Lessard formulation} is precisely the condition derived in \cite[Theorem 4]{lessard2016analysis} for the case of strongly convex objective functions, time-invariant first-order algorithms, and \emph{pointwise} IQCs. 
\end{remark}
\medskip

Having established the main result, it now remains to determine the matrices $M_k^{i}, \ i \in \{0,1,2,3\}$ that construct the LMI in \eqref{thm: main 1}. To this end, we first need to introduce IQCs in the context of optimization algorithms.

\subsection{IQCs for optimization algorithms} \label{sec: Integral Quadratic Constraints for optimization algorithms}
In control theory, there are various approaches and criteria for stability of linear dynamical systems in feedback interconnection with a memoryless and possibly time-varying nonlinearity. In this context, IQCs, originally proposed by Megretski and Rantzer \cite{megretski1997system}, is a powerful tool for describing various classes of nonlinearities, and are particularly useful for LMI-based stability analysis. Lessard et al. \cite{lessard2016analysis} have recently adapted the theory of IQCs for use in optimization algorithms. Specifically, they translate the first-order defining properties of convex functions into various forms of IQCs for their gradient mappings. In the following, we briefly describe the notion of pointwise IQCs \cite{lessard2016analysis} (or quadratic constraints), that will be essential for subsequent developments.

\subsubsection{Pointwise IQCs} Consider a mapping $\phi \colon \! \mathbb{R}^d \to \mathbb{R}^d$ and a chosen ``reference" input-output pair\footnote{As we will see later, the reference point is chosen as the fixed point of the interconnected system we wish to analyze.} $(x_{\star},\phi(x_{\star}))$, $x_{\star} \in \dom \, \phi$. We say that $\phi$ satisfies the \emph{pointwise} IQC defined by $(Q_{\phi},x_{\star},\phi(x_{\star}))$ on $\mathcal{S}\subseteq \dom \, \phi$ if for all $x \in \mathcal{S}$, the following inequality holds \cite{lessard2016analysis}.
%
\begin{align} \label{eq: IQC}
\begin{bmatrix}
x-x_{\star} \\ \phi(x)-\phi(x_{\star})
\end{bmatrix}^\top  Q_{\phi} \, \begin{bmatrix}
x-x_{\star} \\ \phi(x)-\phi(x_{\star})
\end{bmatrix}  \geq 0,
\end{align}
where $Q_{\phi} \in \mathbb{S}^{2d}$ is a symmetric, indefinite matrix.\footnote{If $Q_{\phi}$ is positive (semi)definite, the quadratic constraint holds trivially and is not informative about $\phi$.}  Many inequalities in optimization can be represented as IQCs of the form \eqref{eq: IQC}. For instance, suppose $\phi(x)$ is $L_{\phi}$-Lipschitz continuous on $\mathcal{S}\subseteq \dom \, \phi$ for some positive and finite $L_{\phi}$,  i.e., $\|\phi(x)-\phi(x_{\star})\|_2 \leq L_{\phi} \|x-x_{\star}\|_2$ for all $(x,x_{\star} )\in \mathcal{S}\times \mathcal{S}$. By squaring both sides and rearranging terms, we obtain
\begin{align} \label{eq: Lipschitz map}
\begin{bmatrix}
x-x_{\star} \\ \phi(x)-\phi(x_{\star})
\end{bmatrix}^\top \begin{bmatrix}
L_{\phi}^2I_d  & 0\\ 0& -I_d
\end{bmatrix} \, \begin{bmatrix}
x-x_{\star} \\ \phi(x)-\phi(x_{\star})
\end{bmatrix} \geq 0,
\end{align}
which equivalently describes Lipschitz continuity. As another example, assume $\phi$ is a \emph{firmly nonexpansive mapping} on $\mathcal{S}$. That is, for all $(x,x_{\star} )\in \mathcal{S}\times \mathcal{S}$, we have that $\|\phi(x)-\phi(x_{\star})\|_2^2 \leq (x-x_{\star})^\top (\phi(x)-\phi(x_{\star}))$. This inequality can be rewritten as
\begin{align} \label{eq: firm nonexpansive map}
\begin{bmatrix}
x-x_{\star} \\ \phi(x)-\phi(x_{\star})
\end{bmatrix}^\top \begin{bmatrix}
0  & \frac{1}{2}I_d\\ \frac{1}{2}I_d & -I_d
\end{bmatrix} \, \begin{bmatrix}
x-x_{\star} \\ \phi(x)-\phi(x_{\star})
\end{bmatrix} \geq 0.
\end{align}
Note that by the Cauchy-Schwartz inequality, firm non-expansiveness implies Lipschitz continuity with Lipschitz parameter equal to one, i.e., \eqref{eq: firm nonexpansive map} implies \eqref{eq: Lipschitz map} with $L_{\phi}=1$.
There are many other interesting properties such as monotonicity (also known as incremental passivity), one-sided Lipschitz continuity, cocoercivity, etc., that could be represented by quadratic constraints. 
In the next subsection, we will focus on the gradient mapping of a convex function from an IQC perspective. 
\subsubsection{IQCs for (strongly) convex functions} \label{sec: IQCs for (strongly) convex functions}
%
%

Consider the gradient mapping $\phi=\nabla f$, where $f \in \mathcal{F}(m_f,L_f)$. It directly follows from the definition of (strong) convexity in \eqref{eq: strong convexity} that, $\nabla f$ satisfies the quadratic constraint
\begin{align} \label{eq: gradient strongly convex}
\begin{bmatrix}
x\!-\!y \\ \nabla f(x)\!-\!\nabla f(y)
\end{bmatrix}^\top \begin{bmatrix}
-m_f I_d & \frac{1}{2} I_d \\ \frac{1}{2} I_d & 0
\end{bmatrix}\begin{bmatrix}
x\!-\!y \\ \nabla f(x)\!-\!\nabla f(y)
\end{bmatrix} \geq 0.
\end{align}
Similarly, the Lipschitz inequality in \eqref{eq: Lipschitz continuity} can be represented as
\begin{align} \label{eq: gradient Lipschitz}
\begin{bmatrix}
x\!-\!y \\ \nabla f(x)\!-\!\nabla f(y)
\end{bmatrix}^\top \begin{bmatrix}
L_f^2 I_d & 0 \\ 0 & -1
\end{bmatrix}\begin{bmatrix}
x\!-\!y \\ \nabla f(x)\!-\!\nabla f(y)
\end{bmatrix} \geq 0.
\end{align}
To combine strong convexity and Lipschitz continuity in a single inequality, we note that $\nabla f$ also satisfies \cite{nesterov2013introductory}
\begin{align} \label{eq: strongly convex lipschitz 2}
\frac{m_f L_f}{m_f\!+\!L_f}\|y\!-\!x\|_2^2 \!+\! \frac{1}{m_f\!+\!L_f}\|\nabla f(y)\!-\!\nabla f(x)\|_2^2 \leq (\nabla f(y)\!-\!\nabla f(x))^\top(y\!-\!x).
\end{align}
The above inequality can be represented by the following quadratic constraint \cite{lessard2016analysis}, 
\begin{align}\label{eq: strongly convex IQC}
\begin{bmatrix}
x\!-\!y \\ \nabla f(x)\!-\!\nabla f(y)
\end{bmatrix}^\top \! Q_f \! \begin{bmatrix}
x\!-\!y \\ \nabla f(x)\!-\!\nabla f(y)
\end{bmatrix} \geq 0,
\ \  Q_f \!= \! \begin{bmatrix}
\frac{-m_fL_f}{m_f+L_f}I_d & \frac{1}{2}I_d\\ \frac{1}{2} I_d& \frac{-1}{m_f+L_f}I_d
\end{bmatrix}.
\end{align}
In the language of IQCs, we can say that the map $\phi = \nabla f$ satisfies the pointwise IQC defined by $(Q_f,x_{\star},\nabla f(x_{\star}))$, where the reference point $x_{\star}=y \in \mathcal{S}$ is arbitrary. Note that \eqref{eq: strongly convex IQC} encapsulates both (strong) convexity and Lipschitz continuity in a single IQC. It turns out that this quadratic constraint is both necessary and sufficient for the inclusion $f \in \mathcal{F}(m_f,L_f)$.

\medskip

\textit{Non-differentiable convex functions.} The above analysis can be extended to non-differentiable convex functions. Formally, the subdifferential $\partial f$ of a convex function $f \colon \mathbb{R}^d \to \mathbb{R} \cup \{+\infty\}$ is defined as
\begin{align} \label{eq: subdifferential}
\partial f(x)=\{\gamma \colon \gamma^\top (y-x) + f(x) \leq f(y),\ \forall y \in \dom \, f\},
\end{align}
where $\gamma$ is any subgradient of $f$, which we denote by $T_f(x)$. Adding the inequality in \eqref{eq: subdifferential} to the same inequality but with $x$ and $y$ interchanged, we obtain
\begin{align*}
(T_f(x)-T_f(y))^\top (x-y) \geq 0,
\end{align*}
which is equivalent to monotonicity of the subdifferential operator. Therefore, any subgradient of $f$ satisfies \eqref{eq: strongly convex IQC} with $L_f=\infty$. Note that this property holds even when $f$ is not convex.
%


\section{Performance results for unconstrained smooth programming} \label{section: unconstrained smooth programming}
In this section, we consider first-order algorithms designed to solve problems of the form
\begin{align} \label{eq: unconstrained optimization}
x_{\star} \in \mathrm{argmin}_{x \in \mathbb{R}^d} \ f(x) \quad \mbox{where}  \quad f \in \mathcal{F}(m_f,L_f).
\end{align}
The well-known optimality condition in this case is
$$\mathcal{X}_{\star} = \{ x_{\star} \in \dom \, f \colon \nabla f(x_{\star})=0\}.$$
We now consider an iterative first-order algorithm in the canonical form \eqref{eq: discrete time dynamics} for solving \eqref{eq: unconstrained optimization}, where the feedback nonlinearity is given by $\phi=\nabla f$. Since the sequences $\{x_k\}$ and $\{y_k\}$ converge to the same fixed point in the optimal set by assumption, we must have that $\nabla f(y_{\star})=\nabla f(x_{\star})=0$. In other words, 
%
%
%
%
the fixed points of \eqref{eq: discrete time dynamics} satisfy
\begin{align} \label{eq: discrete time dynamics fixed points}
\xi_{\star} = A_k \xi_{\star}, \quad  y_\star = C_k \xi_{\star}, \quad  u_{\star} =\nabla f(y_\star) = 0, \quad x_{\star}&=E_k \xi_{\star}=y_{\star}, \quad \mbox{for all } k.
\end{align}
In the following result, we characterize the quadratic bounds in \eqref{thm: main result 2} for the class $\mathcal{F}(m_f,L_f)$.
\begin{lemma}\label{prop: Quadratic bounds for convex functions}
	Let $x_{\star} \in \mathrm{argmin}_{x \in \mathbb{R}^d} \ f(x)$ be a minimizer of $f \in \mathcal{F}(m_f,L_f)$ with a finite optimal value $f(x_{\star})$. Consider an iterative first-order algorithm in the state-space form \eqref{eq: discrete time dynamics} with $\phi=\nabla f$, where the fixed points $(\xi_{\star},u_{\star},y_{\star},x_{\star})$ satisfy 
		\begin{align} 
		\xi_{\star} = A_k \xi_{\star}, \ \ y_\star = C_k \xi_{\star}, \ \ u_{\star} =\nabla f(y_{\star})=0, \ \ x_{\star}=E_k \xi_{\star}=y_{\star} \quad \mbox{for all } k.
		\end{align}
		Define $e_k = [(\xi_k-\xi_{\star})^\top \ (u_k-u_{\star})^\top]^\top$. Then the following inequalities hold for all $k$.
	\begin{subequations} \label{eq: quadratic bounds}
		\begin{align} 
			f(x_{k+1}) - f(x_k) &\leq e_k^\top M_k^{1} e_k, \label{thm: main 3}\\
			f(x_{k+1}) - f(x_{\star}) &\leq e_k^\top M_k^{2}e_k, \label{thm: main 4}\\
			0 &\leq e_k^\top M_k^{3} e_k, \label{thm: main 5}
		\end{align}
	\end{subequations}
	where $M_k^1,M_k^2,M_k^3$ are given by
	\begin{align} \label{lemma: LMI matrices convex functions 1}
		M_k^{1} = N_k^{1} + N_k^{2}, \quad M_k^{2}= N_k^{1} + N_k^{3}, \quad M_k^{3} = N_k^{4}.
	\end{align}
	with 
	\begin{align*} 
		N_k^{1} \!&= \begin{bmatrix}
			E_{k+1} A_k\!-\! C_k & E_{k+1}B_k \\ 0 & I_d 
		\end{bmatrix}^\top
		\begin{bmatrix}
			\frac{L_f}{2} I_d & \frac{1}{2} I_d \\ \frac{1}{2} I_d & 0
		\end{bmatrix}
		\begin{bmatrix}
			E_{k+1} A_k\!-\! C_k & E_{k+1}B_k \\ 0 & I_d 
		\end{bmatrix} , \\ \nonumber
		N_k^{2} &=\begin{bmatrix}
			C_k-E_k & 0 \\ 0 & I_d
		\end{bmatrix}^\top \begin{bmatrix}
			-\frac{m_f}{2} I_d & \frac{1}{2} I_d \\ \frac{1}{2} I_d & 0
		\end{bmatrix} \begin{bmatrix}
			C_k-E_k & 0 \\ 0 & I_d
		\end{bmatrix},  \\ \nonumber 
		N_k^{3} &= \begin{bmatrix}
			C_k & 0 \\ 0 & I_d
		\end{bmatrix}^\top \begin{bmatrix}
			-\frac{m_f}{2} I_d & \frac{1}{2} I_d \\ \frac{1}{2} I_d & 0
		\end{bmatrix} \begin{bmatrix}
			C_k & 0 \\ 0 & I_d
		\end{bmatrix}, \\
		N_k^{4} &=\begin{bmatrix}
			C_k & 0 \\ 0 & I_d
		\end{bmatrix}^\top \begin{bmatrix}
			\frac{-m_fL_f}{m_f+L_f} I_d & \frac{1}{2} I_d \\ \frac{1}{2} I_d  & \frac{-1}{m_f+L_f} I_d 
		\end{bmatrix} \begin{bmatrix}
			C_k & 0 \\ 0 & I_d
		\end{bmatrix}.
	\end{align*} 
    \end{lemma}
\begin{proof}
	First, by Lipschitz continuity of $\nabla f$, we can write
	\begin{align} \label{lemma: LMI matrices convex functions 3}
	f(x_{k+1}) - f(y_k) & \leq \begin{bmatrix}
	x_{k+1} -y_k \\ \nabla f(y_k) 
	\end{bmatrix}^\top \begin{bmatrix}
	\frac{L_f}{2} I_d & \frac{1}{2} I_d \\ \frac{1}{2} I_d & 0
	\end{bmatrix}\begin{bmatrix}
	x_{k+1} -y_k \\ \nabla f(y_k) 
	\end{bmatrix}.
	\end{align}
	From the recursion in \eqref{eq: discrete time dynamics}, we have that
	\begin{align}  \label{lemma: LMI matrices convex functions 4}
	\begin{bmatrix}
	x_{k+1} -y_k \\ \nabla f(y_k) 
	\end{bmatrix}= \begin{bmatrix}
	E_{k+1} A_k- C_k & E_{k+1}B_k \\ 0 & I_d 
	\end{bmatrix} 	\begin{bmatrix}
	\xi_{k} -\xi_{\star} \\ u_k - u_{\star}
	\end{bmatrix}.
	\end{align}
	Substituting \eqref{lemma: LMI matrices convex functions 4} in \eqref{lemma: LMI matrices convex functions 3} yields
	\begin{align} \label{lemma: LMI matrices convex functions 5}
	f(x_{k+1}) - f(y_k) & \leq e_k^\top 
	N_k^{1}
	e_k.
	\end{align}
	Next, we use (strong) convexity and the identity $y_k-x_k=(C_k-E_k)(\xi_{k}-\xi_{\star})$ to write
	\begin{align} \label{lemma: LMI matrices convex functions 6}
	f(y_k) - f(x_k) & \leq \begin{bmatrix}
	y_{k} -x_k \\ \nabla f(y_k) 
	\end{bmatrix}^\top \begin{bmatrix}
	-\frac{m_f}{2} I_d & \frac{1}{2} I_d \\ \frac{1}{2} I_d & 0
	\end{bmatrix}\begin{bmatrix}
	y_{k} -x_k \\ \nabla f(y_k) 
	\end{bmatrix}\\
	&\leq e_k^\top \begin{bmatrix}
	C_k-E_k& 0 \\ 0 & I_d
	\end{bmatrix}^\top  \begin{bmatrix}
	-\frac{m_f}{2} I_d & \frac{1}{2} I_d \\ \frac{1}{2} I_d & 0
	\end{bmatrix} \begin{bmatrix}
	C_k-E_k & 0 \\ 0 & I_d
	\end{bmatrix} e_k  \nonumber \\
	& = e_k^\top N_k^{2} e_k. \nonumber 
	\end{align}
	Adding both sides of \eqref{lemma: LMI matrices convex functions 5} and \eqref{lemma: LMI matrices convex functions 6} yields
	\begin{align*}
	f(x_{k+1}) - f(x_k) \leq e_k^\top (N_k^{1}+N_k^{2}) e_k = e_k^\top M_k^{1} e_k.
	\end{align*}
	By (strong) convexity and the identity $y_k-y_{\star}=C_k(\xi_{k}-\xi_{\star})$, we can write
	\begin{align} \label{lemma: LMI matrices convex functions 8}
	f(y_k) - f(y_{\star}) &\leq \begin{bmatrix}
	y_k-y_{\star} \\ \nabla f(y_k)
	\end{bmatrix}^\top \begin{bmatrix}
	-\frac{m_f}{2} I_d & \frac{1}{2} I_d \\ \frac{1}{2} I_d & 0
	\end{bmatrix} \, \begin{bmatrix}
	y_k-y_{\star} \\ \nabla f(y_k)
	\end{bmatrix} \\ &=e_k^\top \begin{bmatrix}
	C_k & 0 \\ 0 & I_d
	\end{bmatrix}^\top \begin{bmatrix}
	-\frac{m_f}{2} I_d & \frac{1}{2} I_d \\ \frac{1}{2} I_d & 0
	\end{bmatrix} \begin{bmatrix}
	C_k & 0 \\ 0 & I_d
	\end{bmatrix} e_k \nonumber \\
	&=e_k^\top {N_k^{3}} e_k. \nonumber 
	\end{align}
	By adding both sides of \eqref{lemma: LMI matrices convex functions 5} and \eqref{lemma: LMI matrices convex functions 8} we obtain
	\begin{align*}
	f(x_{k+1}) - f(x_{\star}) \leq e_k^\top (N_k^{1}+N_k^{3}) e_k = e_k^\top M_k^2 e_k.
	\end{align*}
	Finally, since $f \in \mathcal{F}(m_f,L_f)$, the gradient function $\nabla f$ satisfies the IQC in \eqref{eq: strongly convex IQC}. Since $y_k-y_{\star}=C_k(\xi_{k}-\xi_{\star})$, we can write
	\begin{align} \label{thm: Discrete-Time Model 12}
	e_k^\top N_k^{4} e_k = e_k^\top \begin{bmatrix}
	C_k & 0 \\ 0 & I_d
	\end{bmatrix}^\top Q_f \begin{bmatrix}
	C_k & 0 \\ 0 & I_d
	\end{bmatrix} \, e_k= \begin{bmatrix}
	y_k-y_{\star}  \\  u_k-u_{\star}
	\end{bmatrix}^\top Q_f \begin{bmatrix}
	y_k-y_{\star} \\ u_k - u_{\star}
	\end{bmatrix} \geq 0.
	\end{align}
	The proof is now complete.
\end{proof}

In Lemma \ref{prop: Quadratic bounds for convex functions}, we have used Lipschitz continuity and strong convexity assumptions to find the matrices in \eqref{eq: quadratic bounds}. Explicitly, $N_k^1$ follows from Lipschitz continuity, while $N_k^2$ and $N_k^3$ are due to strong convexity. Finally, the matrix $M_k^3 = N_k^4$ describes the quadratic constraints between the input-output pairs $(\xi_{k},u_k)$ that are related via $u_k=\nabla f(C_k \xi_k)$. Note that $M_k^3=N_k^4$ is an indefinite matrix as required.

\medskip

\begin{remark}[Exploiting block diagonal structure] \label{remark: structure}
	\normalfont We can often exploit some special structure in the data matrices $(A_k,B_k,C_k,E_k)$ {to reduce the dimension of the LMI \eqref{thm: main 1}. For many algorithms, the matrices $(A_k,B_k,C_k,E_k)$ are in the form $(A_k=\bar{A}_k \otimes I_d, B_k=\bar{B}_k \otimes I_d, C_k = \bar{C}_k \otimes I_d, E_k=\bar{E}_k \otimes I_d)$ where $(\bar{A}_k,\bar{B}_k,\bar{C}_k,\bar{E}_k)$ are lower dimensional matrices independent of $d$ \cite[$\S$4.2]{lessard2016analysis}}. By selecting $P_k=\bar{P}_k \otimes I_d$, where $\bar{P}_k$ is a lower dimensional matrix, we can factor out all the Kronecker products $\otimes I_d$ from the matrices $M_k^0,M_k^1,M_k^2,M_k^3$ and make the dimension of the corresponding LMI \eqref{thm: main 1} independent of $d$. In particular, a multi-step method with $r \geq 1$ steps yields an $(r+1) \times (r+1)$ LMI. For instance, the gradient method ($r=1$) and the Nesterov's accelerated method ($r=2$) yield $2 \times 2$ and $3 \times 3$ LMIs, respectively. We will use this dimensionality reduction in the forthcoming case studies.
\end{remark}


%

We can now use Lemma \ref{prop: Quadratic bounds for convex functions} in tandem with Theorem \ref{thm: main result} to derive convergence rates for some existing algorithms in the literature.

\subsection{Symbolic rate bounds}

In order to certify a convergence rate for a given algorithm, we must first represent the algorithm in the canonical form \eqref{eq: discrete time dynamics} and obtain the matrices $M_k^1, M_k^2, M_k^3$ that characterize the bounds in \eqref{thm: main result 2}. These matrices are provided in Lemma \ref{prop: Quadratic bounds for convex functions} for the case $f \in \mathcal{F}(m_f,L_f)$. Then, we must formulate the LMI \eqref{thm: main 1} and search for a feasible triple $(a_k,P_k,\sigma_k)$. In view of \eqref{eq: convergence rate discrete}, we seek to find the fastest convergence rate, i.e., the fastest growing $\{a_k\}$. In what follows, we illustrate this approach via analyzing the gradient method and the Nesterov's accelerated algorithm.

\subsubsection{The gradient method} 
Consider the gradient method applied to $ f \in \mathcal{F}(m_f,L_f)$ with constant stepsize:
\begin{align} \label{eq: gradient_method_discrete}
x_{k+1} = x_{k} - h\nabla f(x_{k}).
\end{align}
This recursion corresponds to the the state-space form \eqref{eq: discrete time dynamics} with $(A_k,B_k,C_k,E_k)=(I_d,-h I_d,I_d,I_d)$. By choosing $P_k=p_k I_d$ ($p_k \geq 0$), we can apply the dimensionality reduction outlined in Remark \ref{remark: structure} and reduce the dimension of the LMI. After dimensionality reduction, the matrices $M_k^{i}, \ i \in \{0,1,2,3\}$ in the LMI \eqref{thm: main 1} read as
%
	\begin{align} \label{eq: data matrices gradient method}
	M_k^{0}&= \begin{bmatrix}
	p_{k+1}\!-\! p_k& -h p_{k+1}\\  -h p_{k+1}& h^2 p_{k+1}
	\end{bmatrix},     \\ \nonumber
	M_k^{1}&\!=\begin{bmatrix}
	0 & 0 \\ 0 &  \frac{1}{2}(L_f h^2-2h)
	\end{bmatrix}, \\ \nonumber 
	M_k^{2}&= \!\begin{bmatrix}
	-\frac{m_f}{2} \!&\! \frac{1}{2} \\  \frac{1}{2} \!&\!  \frac{1}{2}(L_f h^2-2h)
	\end{bmatrix}, \\ \nonumber
	M_k^{3}&= \begin{bmatrix}
	\frac{-m_fL_f}{m_f+L_f} & \frac{1}{2} \\ \frac{1}{2}  & \frac{-1}{m_f+L_f}
	\end{bmatrix}.
	\end{align}
%
%
%
%
%
We first consider strongly convex functions ($m_f>0$) for which we make two parameter selections, as follows.
\begin{itemize}[leftmargin=*]
	\item By setting $p_k=\sigma_k=0$, we obtain the LMI
	\begin{align*}
	\begin{bmatrix}
	-\frac{m_f}{2}(a_{k+1}-a_k) & \frac{1}{2}(a_{k+1}-a_{k})  \\ \frac{1}{2}(a_{k+1}-a_{k})  &(\frac{L_fh^2}{2}-h)a_{k+1}\!
	\end{bmatrix} \!\preceq 0 \quad \mbox{for all} \ k.
	\end{align*}
	It is easy to verify that this matrix inequality is equivalent to the conditions $a_{k+1} \leq \rho^{-1} a_k$ and $0 \leq h \leq 2/L_f$, where $\rho = 1+m_f(L_fh^2-2h)$. Solving for $a_k$ and substituting all the parameters in \eqref{eq: Lyapunov drift}, we conclude the following convergence rate for strongly convex functions:
	\begin{align*} 
	f(x_k) - f(x_{\star}) \leq \left(1+m_f(L_fh^2-2h)\right)^k (f(x_0) - f(x_{\star})),  \quad 0 \leq h \leq \dfrac{2}{L_f}.
	\end{align*}
	Notice that the decay rate $\rho$ obeys $0 \leq \rho \leq 1$ as $h$ varies on $[0,2/L_f]$. In particular, by optimizing $\rho$ over $h$, we obtain the optimal step size $h=1/L_f$, yielding the decay rate $\rho=1-m_f/L_f$. 
	\item  By the parameter selection $a_k \equiv 0$ and $p_k=\rho^{-2k}p_0$ $\sigma_k=\lambda_0 \rho^{-2k-2}$, the LMI simplifies to
	\begin{align} \label{eq: gradient method SC LMI}
	\begin{bmatrix}
	p_{0}\!-\! \rho^2 p_0& -h p_{0}\\  -h p_{0}& h^2 p_{0}
	\end{bmatrix} + \lambda_0 \begin{bmatrix}
	\frac{-m_fL_f}{m_f+L_f} & \frac{1}{2} \\ \frac{1}{2}  & \frac{-1}{m_f+L_f}
	\end{bmatrix} \preceq 0,
	\end{align}
	which is the same LMI as the one proposed in \cite{lessard2016analysis} and yields the optimal decay rate $\rho=\max(|1-hm_f|,|1-hL_f|)$.
\end{itemize}
\medskip
We now consider convex functions ($m_f=0$). By the particular selection $p_k=p$ and $\sigma_k=\sigma$, the LMI \eqref{thm: main 1} reduces to
\begin{align} \label{eq: gradient descent weakly convex}
\begin{bmatrix}
0 & \frac{1}{2}(a_{k+1}-a_{k}-2p h+\sigma)  \\ \frac{1}{2}(a_{k+1}-a_{k}-2p h+\sigma)  & (\frac{L_fh^2}{2}-h)a_{k+1}+p h^2-\frac{\sigma}{L_f}
\end{bmatrix}\preceq 0 \quad \mbox{for all} \ k,
\end{align}
which is homogeneous in $(a_k,a_{k+1},p,\sigma)$. We can therefore assume $p=1$, without loss of generality. With this selection, the above LMI becomes equivalent to the following inequalities.
\begin{gather*}
a_{k+1}=a_{k}+2h-\sigma, \quad (\frac{L_fh^2}{2} - h)a_{k+1} + h^2-\dfrac{\sigma}{L_f} \leq 0 \quad \mbox{for all} \ k.
\end{gather*}
Assuming $a_0=0$ and solving for the fastest growing ${a_k}$ that satisfies the above constraints, we obtain the following rate bound:
%
\begin{subequations}\label{eq: gradient descent weakly convex 1}
\begin{align} 
f(x_{k}) - f(x_{\star}) \leq \dfrac{L_f\|x_0-x_{\star}\|_2^2}{Ck},
\end{align}
where $C$ is given by
\begin{align}
C= \begin{cases}
2L_fh & \mbox{for} \ 0 \leq L_fh \leq 1 \\
\dfrac{2(L_fh)^2(2-L_fh)}{(L_fh)^2-2L_fh+2} & \mbox{for} \ 1 \leq L_fh \leq 2
\end{cases}.
\end{align}
\end{subequations}
%
We have provided the detailed derivations in Appendix \ref{app: Symbolic convergence rates for the gradient method}. 
\subsubsection{Nesterov's accelerated method} \label{subsection: Nesterov's Accelerated method}
We now analyze Nesterov's accelerated method \cite{nesterov1983method} applied to $f \in \mathcal{F}(m_f,L_f)$, which consists in the following recursions:
%
	\begin{align}  \label{eq: accelerated gradient method}
	x_{k+1} & = y_k - h \nabla f(y_k), \\
	y_{k} &= x_k + \beta_k (x_k-x_{k-1}),  \nonumber
	\end{align}
%
where $\beta_k \geq 0$ is the momentum coefficient, and $h>0$ is the step size. With an appropriate tuning, this method exhibits an $\mathcal{O}(1/k^2)$ convergence rate when $m_f=0$. 
%
One such tuning is \cite{nesterov1983method,beck2009fast}
\begin{align} \label{eq: Nesterov parameters weakly convex}
\beta_k &= t_k^{-1}(t_{k-1}-1), \quad t_k = \frac{1}{2}(1+\sqrt{1+4t_{k-1}^2}),\quad t_{-1}=1, \quad 0 < h \leq L_f^{-1}.
\end{align}
Notice that by this selection, we can verify that $t_k^2-t_{k-1}^2=t_k$ and $t_{k-1} \geq (k+2)/2$. By defining the state vector $\xi_k = [x_{k-1}^\top \ x_k^\top]^\top$, we can write \eqref{eq: accelerated gradient method} in the canonical form
%
\begin{align}  \label{eq: accelerated gradient method 1}
\xi_{k+1} &= \begin{bmatrix}
0 & I_d \\ -\beta_k I_d & (1+\beta_k) I_d
\end{bmatrix}\xi_{k} + \begin{bmatrix}
0 \\ -h I_d
\end{bmatrix} \nabla f(y_k), \\
y_k &= \begin{bmatrix}
-\beta_k & (1+\beta_k) I_d
\end{bmatrix}\xi_{k}, \nonumber \\
x_k &=\begin{bmatrix}
0 & 1
\end{bmatrix}\xi_k. \nonumber 
\end{align}
%
The fixed points of \eqref{eq: accelerated gradient method 1} are $(\xi_{\star},u_{\star},y_{\star},x_{\star})=([x_{\star}^\top \ x_{\star}^\top]^\top , 0, x_{\star},x_{\star})$, where $x_{\star} \in \mathcal{X}_{\star}$ is any optimal solution to \eqref{eq: unconstrained optimization}.
%
%
Making use of Lemma \ref{prop: Quadratic bounds for convex functions}, the matrices $M_k^i \ i \in \{0,1,2,3\}$ for Nesterov's accelerated method read as
\begin{align} \label{eq: Nesterov method LMI matrices}
M_k^{0}&=  \begin{bmatrix}
A_k^\top P_{k+1} A_k \!-\! P_k& A_k^\top P_{k+1} B_k\\  B_k^\top P_{k+1} A_k& B_k^\top P_{k+1} B_k 
\end{bmatrix}, \\ \nonumber
M_k^{1}&= \begin{bmatrix}
-\frac{1}{2}m_f \beta_k^2 & \frac{1}{2}m_f\beta_k^2 & -\frac{1}{2}\beta_k \\  \frac{1}{2}m_f\beta_k^2 &  -\frac{1}{2}m_f \beta_k^2 & \frac{1}{2}\beta_k\\ -\frac{1}{2}\beta_k & \frac{1}{2}\beta_k & \frac{1}{2}L_f h^2-h
\end{bmatrix}, \\ \nonumber
M_k^{2}&= \begin{bmatrix}
-\frac{1}{2}m_f\beta_k^2 & \frac{1}{2}m_f\beta_k(\beta_k+1) & -\frac{1}{2}\beta_k \\  \frac{1}{2}m_f\beta_k(\beta_k+1) &  -\frac{1}{2}m_f(\beta_k+1)^2 & \frac{1}{2}(\beta_k+1)  \\ -\frac{1}{2}\beta_k & \frac{1}{2}(\beta_k+1) & \frac{1}{2}L_f h^2-h
\end{bmatrix}, \\ \nonumber
M_k^{3} &=\begin{bmatrix}
-\beta_kI_d & 0 \\ (1+\beta_k)I_d & 0 \\ 0 & I_d
\end{bmatrix} \begin{bmatrix}
\frac{-m_fL_f}{m_f+L_f} & \frac{1}{2} \\ \frac{1}{2}  & \frac{-1}{m_f+L_f}
\end{bmatrix}\begin{bmatrix}
-\beta_kI_d  & (1+\beta_k)I_d & 0 \\ 0 & 0 & I_d
\end{bmatrix}.
\end{align}

%
We consider convex settings ($m_f=0$). It is straightforward to verify that for the parameter selection $\sigma_k=0$, $a_k=t_{k-1}^2$ (with $a_0=1$) , and
\begin{align*}
P_k = \dfrac{1}{2h} \begin{bmatrix}
1-t_{k-1} \\ t_{k-1}
\end{bmatrix} \begin{bmatrix}
1-t_{k-1} & t_{k-1}
\end{bmatrix},
\end{align*}
the LMI \eqref{thm: main 1} holds with equality, i.e., all the entries of the matrix is zero. Therefore, Theorem \ref{thm: main result} implies
\begin{align} \label{eq: Nesterov weakly convex rate bound}
f(x_k) - f(x_{\star}) \leq \dfrac{f(x_0)-f(x_{\star})+\frac{1}{2h}\|x_0-x_{\star}\|_2^2}{t_{k-1}^2} = \mathcal{O}(\dfrac{1}{k^2}).
\end{align}
where the equality follows from the fact that $t_{k-1} \geq (k+2)/2$.

The analysis of Nesterov's method shows that finding a symbolic feasible pair $(a_k,P_k)$ to the LMI \eqref{thm: main 1} can be subtle. Nevertheless, we can also search for these parameters via a numerical scheme, as we describe next.

\begin{figure}[htbp]
	\centering
	\includegraphics[width=0.7\textwidth]{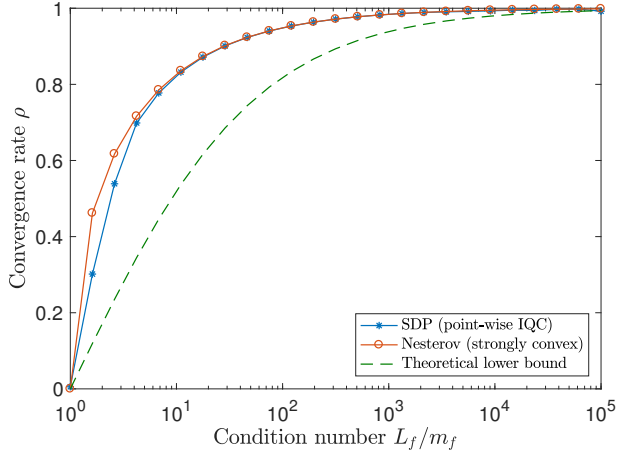}
	\caption{\small Comparison of rate bounds in Nesterov's method for different ratios $\kappa_f = L_f/m_f$ using the parameter selection $h=1/L_f$ and $\beta=\frac{\sqrt{\kappa_f}-1}{\sqrt{\kappa_f}+1}$. For this parameter selection, the analytical rate bound is $\rho = \sqrt{1-\frac{1}{ \sqrt{ \kappa_f }}}$ \cite{nesterov2013introductory}.}
	\label{fig: Nesterov_strongly_convex}
\end{figure}

\subsection{Numerical bounds for exponential rates} \label{sec: Numerical bounds for exponential rates}

We could also use the results of Theorem \ref{thm: main result} to search for the parameters $(a_k,P_k)$ numerically. This approach is particularly efficient for time-invariant algorithms  with exponential convergence. Under these assumptions, the sequence of LMIs in \eqref{thm: main 1} collapses into the single LMI in \eqref{eq: LMI exponential}, which no longer depends on the iteration index $k$. We can then use this LMI to find the exponential decay rate numerically. Explicitly, the matrix inequality \eqref{eq: LMI exponential} is an LMI in $(a_0,P_0,\lambda_0)$ for a fixed $\rho^2$. We can therefore use a bisection search aiming to find the smallest value of the convergence rate $\rho$ that satisfies \eqref{eq: LMI exponential} for some $(a_0,P_0,\lambda_0)$. Notice that the LMI in \eqref{eq: LMI exponential} is homogeneous in its decision variables. We can therefore assume $\lambda_0=1$, without loss of generality.


%

In Figure \ref{fig: Nesterov_strongly_convex}, we compare the numerical rate bounds with the theoretical lower bound and the analytical rate bound of Nesterov's method with the parameter selection $h=1/L_f$ and $\beta={(\sqrt{\kappa_f}-1)}/{(\sqrt{\kappa_f}+1)}$ \cite{nesterov2013introductory}. We observe that, the SDP yields slightly better bounds than the analytical rate bound. 

We remark that, in \cite{lessard2016analysis} the authors make use of quadratic Lyapunov functions and ``off-by-one" IQCs to obtain numerical rate bounds for strongly convex problems. They have shown that pointwise IQCs alone exhibit crude bounds and the use of off-by-one IQCs improve the numerical solutions greatly. In contrast, we have utilized \emph{nonquadratic} Lyapunov functions and \emph{pointwise} IQCs, which yield nonconservative rate bounds. This nonconservatism is due to the inclusion of the term $a_k(F(x_k)-F(x_{\star}))$ in the Lyapunov function. We conjecture that, by using off-by-one IQCs or other IQCs developed in \cite{lessard2016analysis} in our Lyapunov framework, we can further improve the numerical bounds.


\subsection{Numerical bounds for subexponential rates} 

For time-varying algorithms and nonstrongly convex functions, the convergence rate is subexponential and the LMI \eqref{thm: main 1} becomes dependent on the iteration number. In this case, a numerical approach amounts to solving an infinite sequence of LMIs to find a rate-generating sequence $\{a_k\}$. Nevertheless, we can truncate the sequence of LMIs in order to obtain rate bounds for a \emph{finite} number of iterations. Specifically, for a given $N>0$, we consider the following SDP:
\begin{align} \label{eq: SDP formulation}
&{\mbox{maximize \ \ }} \ a_{N} \\ &\mbox{subject to} \quad \mbox{for $k=0,1,\cdots,N-1$:} \nonumber \\ 
&\qquad \qquad \quad \ M_k^{0} + a_{k}M_k^{1} + (a_{k+1} - a_k) M_k^{2} +\sigma_k M_k^{3} \preceq 0, \nonumber  \\
& \quad \qquad \quad \quad \ a_{k+1}\geq a_k, \quad \sigma_k \geq 0, \quad P_k \succeq 0. \nonumber
\end{align}
with decision variables $\{(a_k,P_k,\sigma_k)\}_{k=1}^{N}$. Denoting the optimal solution of \eqref{eq: SDP formulation} by $a^\star_N$, Theorem \ref{thm: main result} immediately implies
\begin{align}
f(x_N) - f(x_{\star}) \leq \dfrac{V_0(x_0,\xi_0)}{a^\star_N}.
\end{align}
In other words, \eqref{eq: SDP formulation} searches for the smallest upper bound on the $N$-th (last) iterate suboptimality, subject to the stability constraint imposed by LMI \eqref{thm: main 1}. Notice that \eqref{eq: SDP formulation} is homogeneous in the decision variables. To get a sensible problem, we must normalize the variables by, for example, requiring all of them to add up to a positive constant. Furthermore, the $k$-th LMI in \eqref{eq: SDP formulation} is a function of $a_k$, $a_{k+1}$, $P_{k}$, $P_{k+1}$, and $\sigma_k$. This implies the SDP is banded with a fixed bandwidth independent of $N$, the number of iterations. We can exploit this sparsity structure in solving the SDP efficiently. For instance, for the Nesterov's method and $N=10^3$ iterations, solving the SDP takes less than 10 seconds to solve with an off-the-shelf solver.

In Figure \ref{fig: Nesterov_weakly_convex}, we plot numerical rate bounds obtained by solving \eqref{eq: SDP formulation} for the Nesterov's accelerated method with the parameter selection given in  \eqref{eq: Nesterov parameters weakly convex}. We also plot the analytical rate bound given in \eqref{eq: Nesterov weakly convex rate bound}. We observe that the numerical rate bound coincides with the analytical rate.

\begin{figure}[htbp]
	\centering
	\includegraphics[width=0.7\textwidth]{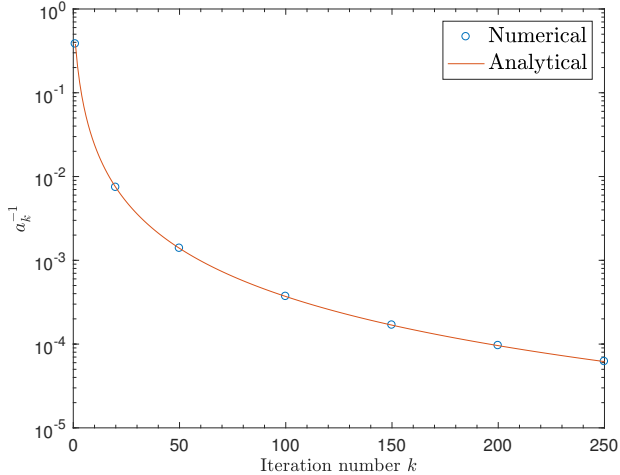}
	\caption{\small Comparison of rate bounds obtained by numerically solving the SDP in \eqref{eq: SDP formulation} and analytical rate bounds for the Nesterov's accelerated method with the parameter selection given in \eqref{eq: Nesterov parameters weakly convex}.}
	\label{fig: Nesterov_weakly_convex}
\end{figure}

%

%
\section{Composite optimization problems} \label{sec: Proximal methods}
%
In this section, we consider composite optimization problems of the form
\begin{align} \label{eq: unconstrained minimization}
\mathcal{X}_{\star} = \mathrm{argmin}_{x \in \mathbb{R}^d} \{F(x) = f(x) + g(x)\},
\end{align}
where $f \colon \mathbb{R}^d \to \mathbb{R}$ is differentiable CCP, while $g \colon \mathbb{R}^d \to \mathbb{R} \cup \{+\infty\}$ is nondifferentiable and CCP. We assume the optimal solution set $\mathcal{X}_{\star}$ is nonempty and closed, and the optimal value $F(x_{\star})$ is finite. Under these assumptions, the optimality condition for \eqref{eq: unconstrained minimization} is given by
\begin{align} \label{eq: first-order optimality}
\mathcal{X}_{\star} = \{x_{\star}\in \dom \, f \cap \dom \, g \colon 0 \in \nabla f(x_{\star})+\partial g(x_{\star})\}.
\end{align}
%
%
Formally, the objective function in \eqref{eq: unconstrained minimization} is nonsmooth and subgradient methods are very slow. Splitting methods such as proximal algorithms circumvent this issue by exploiting the special structure of the objective function to achieve comparable convergence rates to their counterparts in smooth programming. In this section, we analyze proximal algorithms using Theorem \ref{thm: main result}. To this end, we first show that we can represent these algorithms in the canonical form \eqref{eq: discrete time dynamics}, where the feedback nonlinearity $\phi$ is the generalized gradient mapping of $F$. By deriving the proximal counterpart of Lemma \ref{prop: Quadratic bounds for convex functions}, we can then immediately apply Theorem \ref{thm: main result} to proximal algorithms.

%
\subsection{Generalized gradient mapping}\label{subsection: IQCs for generalized gradient mapping} 
Let $g \colon \mathbb{R}^d \to \mathbb{R} \cup \{+\infty\}$ be a CCP function. The proximal operator $\Pi_{g,h} \colon \mathbb{R}^d \to \mathbb{R}^d$ of $g$ with parameter $h>0$ is defined as
\begin{align} \label{eq: proximal operator 1}
\Pi_{g,h}(x)=\mathrm{argmin}_{y \in \mathbb{R}^d} \ \{g(y) + \dfrac{1}{2h} \|y-x\|_2^2\}.
\end{align}
%
For the composite function in \eqref{eq: unconstrained minimization}, we define the generalized gradient mapping $\phi_{h} \colon \mathbb{R}^d \to \mathbb{R}^d$ as
\begin{align} \label{eq: generalized gradient 0}
\phi_h(x)=\dfrac{1}{h}(x-\Pi_{g,h}(x - h \nabla f(x))), \ h>0.
\end{align} 
with $\dom \, \phi_h = \dom \, f$. Notice that when $g(x)\equiv 0$ (so that $\Pi_{g,h}(x)=x$), the generalized gradient mapping simplifies to the gradient function $\nabla f$. Furthermore, we have that $\phi_h(x_{\star})=0$ for $x_{\star} \in \mathcal{X}_{\star}$, i.e., $\phi_h$ vanishes at optimality. In the following proposition, we characterize several properties of $\phi_h$, which will prove useful.
\begin{proposition}\label{lemma: IQCs for generalized gradient mapping}
	Consider the composite function $F=f+g$ with $f \! \in \! \mathcal{F}(m_f,L_f)$ and $g \in \mathcal{F}(0,\infty)$. Correspondingly, define the generalized gradient mapping $\phi_{h}$ of $F$ as in \eqref{eq: generalized gradient 0} .
	\begin{enumerate}[leftmargin=*]
		\medskip
		\item $\phi_h$ satisfies the pointwise IQC defined by $(Q_{\phi_h},x_{\star},\phi_h(x_{\star}))$, where $Q_{\phi_h}$ is given by
		\begin{align} \label{eq: generalized gradient IQC}
		Q_{\phi_h} = \begin{bmatrix}
		\dfrac{1}{2h}(\gamma_f^2-1)& \dfrac{1}{2} \\ \dfrac{1}{2} & -\dfrac{h}{2}
		\end{bmatrix} \! \otimes I_d,
		%
		\end{align}
		with $\gamma_f=\max \{|1-h L_f|,|1- h m_f|\}$.
		\medskip
		\item The following inequality
		\begin{align} \label{eq: composite function inequality}
		F(y\!-\!h \phi_h(y))\!-\!F(x) \leq & \phi_h(y)^\top (y\!-\!x)\!-\! \dfrac{m_f}{2} \|y\!-\!x\|_2^2+(\frac{1}{2}L_f h^2\!-\!h) \| \phi_h(y)\|_2^2,
		\end{align}
		holds for all $h \geq 0$ and $x,y \in \dom \, F$.
		\medskip
		\item $\phi_h(x_{\star})=0$ if and only if $x_{\star} \in \mathrm{argmin} \ F(x)$.
	\end{enumerate}
\end{proposition}
\begin{proof}
	See Appendix \ref{lemma: IQCs for generalized gradient mapping proof}.
\end{proof}

\subsection{Proximal algorithms} Using the definition of generalized gradient mapping in \eqref{eq: generalized gradient 0}, we can represent proximal algorithms with the same state-space structure as in \eqref{eq: discrete time dynamics}, where the feedback nonlinearity is $\phi=\phi_h$. For example, the Nesterov's accelerated proximal gradient method is defined by
\begin{align} \label{eq: accelerated proximal gradient method 0}
x_{k+1} &= \Pi_{g,h}(y_k-h\nabla f(y_k)), \\
y_k &= x_k + \beta_k (x_k-x_{k-1}), \nonumber
\end{align}
%
which, by using \eqref{eq: generalized gradient 0}, can be rewritten as
\begin{align} \label{eq: accelerated proximal gradient method 00}
x_{k+1} &= x_k + \beta_k (x_k-x_{k-1}) -h \phi_h(y_k), \\
y_k &= x_k + \beta_k (x_k-x_{k-1}). \nonumber
\end{align}
By defining the state vector $\xi_k = [x_{k-1}^\top \ x_k^\top]^\top \in \mathbb{R}^{2d}$, the corresponding state-space matrices $(A_k,B_k,C_k)$ are given by
\begin{align} \label{eq: proximal nesterov matrices}
\left[
\begin{array}{c|c}
A_k & B_k \\
\hline
C_k & 0
\end{array}
\right] = \left[
\begin{array}{c|c}
\begin{matrix}
0 &I_d \\ -\beta_kI_d & (\beta_k+1)I_d
\end{matrix} & \begin{matrix} 0 \\ -hI_d\end{matrix} \\
\hline
\begin{matrix} -\beta_kI_d & (\beta_k+1)I_d\end{matrix} & 0
\end{array}
\right].
\end{align}
Recall the assumption that the sequences $\{x_k\}$ and $\{y_k\}$ converge to the same fixed point in the optimal set. Since $\phi_h$ is zero at optimality, we must therefore have that $\phi_{h}(y_{\star})=\phi_{h}(x_{\star})=0$. In other words, 
%
%
%
%
the fixed points satisfy
\begin{align} \label{eq: discrete time dynamics proximal fixed points}
\xi_{\star} = A_k \xi_{\star}, \quad  y_\star = C_k \xi_{\star}, \quad  u_{\star} =\phi_{h}(y_\star) = 0, \quad x_{\star}&=E_k \xi_{\star}=y_{\star}, \quad \mbox{for all } k.
\end{align}
Having characterized the generalized gradient mapping with quadratic constraints, we are now ready to develop the proximal counterpart of Lemma \ref{prop: Quadratic bounds for convex functions}.
\begin{lemma}\label{prop: Quadratic bounds for composite convex functions}
	Let $x_{\star} \in \mathrm{argmin} \ F(x)$ be a minimizer of $F=f+g$ with a finite optimal value $F(x_{\star})$, where $f \in \mathcal{F}(m_f,L_f)$ and $g \in \mathcal{F}(0,\infty)$. Consider a proximal first-order algorithm in the state-space form \eqref{eq: discrete time dynamics} with $\phi=\phi_h$ defined as in \eqref{eq: generalized gradient 0}. Suppose the fixed points $(\xi_{\star},u_{\star},y_{\star},$ $x_{\star})$ satisfy 
	\begin{align} 
	\xi_{\star} = A_k \xi_{\star}, \ \ y_\star = C_k \xi_{\star}, \ \ u_{\star} =\phi_{h}(y_{\star})=0, \ \ x_{\star}=E_k \xi_{\star}=y_{\star} \quad \mbox{for all } k.
	\end{align}
	Then the following inequalities hold for all $k$.
	\begin{subequations} \label{eq: quadratic bounds composite}
		\begin{align} 
			F(x_{k+1}) - F(x_k) &\leq e_k^\top M_k^{1} e_k, \label{eq: quadratic bounds composite 1}\\
			F(x_{k+1}) - F(x_{\star}) &\leq e_k^\top M_k^{2}e_k, \label{eq: quadratic bounds composite 2}\\
			0 &\leq e_k^\top M_k^{3} e_k, \label{eq: quadratic bounds composite 3}
		\end{align}
	\end{subequations}
	where $e_k = [(\xi_k-\xi_{\star})^\top \ (u_k-u_{\star})^\top]^\top$ and $M_k^1, M_k^2,M_k^3$ are given by
	\begin{align} \label{lemma: LMI matrices convex composite functions 1}
	M_k^{1} \!&=\begin{bmatrix}
	C_k\!-\!E_k & 0 \\ 0 & I_d
	\end{bmatrix}^\top \begin{bmatrix}
	-\frac{m_f}{2} & \frac{1}{2}\\ \frac{1}{2}& (\frac{1}{2}L_fh^2\!-\!h)
	\end{bmatrix}  \begin{bmatrix}
	C_k\!-\!E_k & 0 \\ 0 & I_d
	\end{bmatrix}, \\ \nonumber
	M_k^{2} &=\begin{bmatrix}
	C_k & 0 \\ 0 & I_d
	\end{bmatrix}^\top \begin{bmatrix}
	-\frac{m_f}{2} & \frac{1}{2}\\ \frac{1}{2}& (\frac{1}{2}L_fh^2\!-\!h)
	\end{bmatrix}  \begin{bmatrix}
	C_k & 0 \\ 0 & I_d
	\end{bmatrix}, \\
	M_k^{3} &=\begin{bmatrix}
	C_k & 0 \\ 0 & I_d
	\end{bmatrix}^\top Q_{\phi_h} \begin{bmatrix}
	C_k & 0 \\ 0 & I_d
	\end{bmatrix}. \nonumber
	\end{align}
\end{lemma}
\begin{proof}
	See Appendix \ref{prop: Quadratic bounds for composite convex functions proof}.
\end{proof}

\begin{remark} \normalfont 
	In \cite{lessard2016analysis}, the authors use a different block diagonal representation of proximal algorithms, in which the linear component is in parallel feedback connections with the gradient function $\nabla f$, as well as the subdifferential operator $\partial g$. Then, each nonlinear block is described by its corresponding IQC, i.e., the IQC of gradient mappings and subdifferential operators. In contrast, we collectively represent all the nonlinearities in a single feedback component (the generalized gradient mapping), whose IQC is given in Lemma \ref{lemma: IQCs for generalized gradient mapping}. 
\end{remark}

In the following, we use Lemma \ref{prop: Quadratic bounds for composite convex functions} in conjunction with Theorem \ref{thm: main result} to analyze the proximal gradient method and the proximal variant of Nesterov's accelerated method.

\subsubsection{Proximal gradient method}
The classical proximal gradient method is defined by the recursion
\begin{align} \label{eq: proximal gradient method 1}
x_{k+1} = \Pi_{h g}(x_k - h \nabla f(x_k)),
\end{align}
which, by using the definition of the generalized gradient mapping in \eqref{eq: generalized gradient 0}, can be written as
\begin{align}
x_{k+1} = x_k - h \phi_h(x_k).
\end{align}
The state-space matrices are therefore given by $(A_k,B_k,C_k,E_k)=(I_d,-hI_d,I_d,I_d)$. By selecting $P_k=p_k I_d, \ p_k \geq 0$, the matrices $ M_k^{i} \  i=0,1,2,3$ are given by
\begin{subequations} \label{eq: data matrices classical proximal}
	\begin{align}
	M_k^{0}&=  \begin{bmatrix}
	p_{k+1}\!-\! p_k& -h  p_{k+1}\\  -h  p_{k+1}& h ^2 p_{k+1}
	\end{bmatrix} \otimes I_d, \\
	M_k^{1}&\!=\! \begin{bmatrix}
	0 \!&\! 0 \\ 0\!&\!  \frac{1}{2}(L_f h ^2\!-\!2h )
	\end{bmatrix} \otimes I_d, \\
	M_k^{2}&= \begin{bmatrix}
	-\frac{1}{2}m_f & \frac{1}{2} \\ \frac{1}{2}&  \frac{1}{2}(L_f h ^2-2h )
	\end{bmatrix}\otimes I_d,  \\
	M_k^{3}&= \begin{bmatrix}
	\frac{1}{2h }(\gamma_f^2-1)& \frac{1}{2} \\ \frac{1}{2} & -\frac{h }{2}
	\end{bmatrix} \! \otimes I_d,
	\end{align}
\end{subequations}
where $\gamma_f = \max \{|1-h  L_f|,|1-h  m_f|$. 

\medskip

\emph{Strongly Convex Case.} We first consider the selection $a_k \equiv 0$ for strongly convex settings. Then the LMI \eqref{eq: data matrices classical proximal} simplifies to
\begin{align*}
\begin{bmatrix}
p_{k+1}\!-\! p_k& -h  p_{k+1}\\  -h  p_{k+1}& h ^2 p_{k+1}
\end{bmatrix} + \sigma_k \begin{bmatrix}
\frac{\gamma_f^2-1}{2h }& \frac{1}{2} \\ \frac{1}{2} & -\frac{h }{2}
\end{bmatrix} \leq 0.
\end{align*}
It can be verified that the above LMI is equivalent to the conditions 
\begin{align*}
\sigma_k/(2h ) \leq p_k/\gamma_f^2, \quad p_{k+1}-p_k \leq \sigma_k(1-\gamma_f^2)/(2h ).
\end{align*}
These two conditions together imply $p_{k+1} \leq p_k/\gamma_f^2$. Therefore, we can write $p_k=\gamma_f^{-2k} p_0, \ p_0>0$. Using the bound \eqref{eq: exponential convergence 0}, we can establish the bound
\begin{align*}
\|x_k\!-\!x_{\star}\|_2^2 \leq \left(\max \{|1\!-\!h  L_f|,|1\!-\!h  m_f|\}\right)^{2k} \|x_0\!-\!x_{\star}\|_2^2.
\end{align*}
%
%
On the other hand, setting $p_k\equiv 0$ in \eqref{eq: data matrices classical proximal} yields the LMI
\begin{align*}
\begin{bmatrix}
-\dfrac{m_f}{2} (a_{k+1}-a_k)& \dfrac{a_{k+1}-a_k}{2} \\ \dfrac{a_{k+1}-a_k}{2} & (\dfrac{L_fh ^2}{2}-h ) a_{k+1}
\end{bmatrix}\preceq 0.
\end{align*}
Omitting the details, we obtain from the above LMI that $a_{k+1} \leq \rho^{-2} a_k$ and $0 \leq h  \leq 2/L_f$, where $\rho^2 = 1+m_f(L_fh ^2-2h )$. Substituting $a_k$ in \eqref{eq: exponential convergence 1} yields the bound
\begin{align*}
F(x_k)\!-\!F(x_{\star}) \leq (1+m_f(L_fh ^2\!-\!2h ))^k (F(x_0)-F(x_{\star})).
\end{align*}
In particular, the optimal decay rate is attained at $h =1/L_f$, and is equal to $\rho=1-m_f/L_f$.

\medskip

\emph{Convex Case.} When the differentiable component of the objective is convex ($m_f=0$), we select $p_k=p>0, \sigma_k=\sigma$ in \eqref{eq: data matrices classical proximal} to arrive at the LMI
\begin{align*}
\begin{bmatrix}
\dfrac{\sigma}{2h }(\gamma_f^2-1) & \dfrac{1}{2}(a_{k+1}-a_k-2ph+\sigma)  \\ \dfrac{1}{2}(a_{k+1}-a_k-2ph+\sigma)  & (\dfrac{L_fh ^2}{2}-h ) a_{k+1}+p h ^2 - \dfrac{\sigma h}{2} \end{bmatrix}\preceq 0.
\end{align*}
To further simplify the LMI, we take $\sigma=0$. Then the LMI enforces that 
\begin{align*}
a_{k+1}=a_k+2ph, \quad a_0\geq 0, \quad (L_fh ^2/2-h )(a_{k+1} )+p h ^2 \leq 0
\end{align*}
Solving for $a_k$ leads to
\begin{align*}
F(x_k)-F(x_{\star}) \leq \dfrac{a_0 (F(x_0)-F(x_{\star}))+p \|x_0-x_{\star}\|_2^2}{a_0+2 p h  k}.
\end{align*}
In particular, if $a_0=0$, then it must hold that $0 \leq h  \leq 1/L_f$, and we recover the convergence result in \cite[Theorem 3.1]{beck2009fast}.

\subsubsection{Accelerated proximal gradient method}

Consider the proximal variant of Nesterov's accelerated method outlined in \eqref{eq: accelerated proximal gradient method 0}, for which the state-space matrices are given in \eqref{eq: proximal nesterov matrices}. Making use of Lemma \ref{prop: Quadratic bounds for composite convex functions}, the matrices $M_k^i \ i \in \{0,1,2,3\}$ read as
\begin{align} \label{eq: Nesterov method proximal LMI matrices}
M_k^{0}&=  \begin{bmatrix}
A_k^\top P_{k+1} A_k \!-\! P_k& A_k^\top P_{k+1} B_k\\  B_k^\top P_{k+1} A_k& B_k^\top P_{k+1} B_k 
\end{bmatrix}, \\ \nonumber
M_k^{1}&= \begin{bmatrix}
-\frac{1}{2}m_f \beta_k^2 & \frac{1}{2}m_f\beta_k^2 & -\frac{1}{2}\beta_k \\  \frac{1}{2}m_f\beta_k^2 &  -\frac{1}{2}m_f \beta_k^2 & \frac{1}{2}\beta_k\\ -\frac{1}{2}\beta_k & \frac{1}{2}\beta_k & \frac{1}{2}L_f h^2-h
\end{bmatrix}, \\ \nonumber
M_k^{2}&= \begin{bmatrix}
-\frac{1}{2}m_f\beta_k^2 & \frac{1}{2}m_f\beta_k(\beta_k+1) & -\frac{1}{2}\beta_k \\  \frac{1}{2}m_f\beta_k(\beta_k+1) &  -\frac{1}{2}m_f(\beta_k+1)^2 & \frac{1}{2}(\beta_k+1)  \\ -\frac{1}{2}\beta_k & \frac{1}{2}(\beta_k+1) & \frac{1}{2}L_f h^2-h
\end{bmatrix}, \\ \nonumber
M_k^{3} &=\begin{bmatrix}
-\beta_kI_d & 0 \\ (1+\beta_k)I_d & 0 \\ 0 & I_d
\end{bmatrix} \begin{bmatrix}
\frac{1}{2h}(\gamma_f^2-1) I_d & \frac{1}{2}I_d \\ \frac{1}{2}I_d & -\frac{h}{2}I_d
\end{bmatrix} \begin{bmatrix}
-\beta_kI_d  & (1+\beta_k)I_d & 0 \\ 0 & 0 & I_d
\end{bmatrix}.
\end{align}
Observe that the matrices $M_k^0, M_k^1$, and $M_k^2$ are precisely the same as those of Nesterov's method without proximal operation. The only difference is in $M_k^3$. As a result, by setting $\sigma_k=0$ (the coefficient of $M_k^3$) in the LMI \eqref{thm: main 1}, the analysis of Nesterov's accelerated method in $\S$\ref{subsection: Nesterov's Accelerated method} immediately applies to the proximal variant \cite{8262759}.


%
%
%

\begin{remark}[Gradient methods with projection]
	\rm For the case that $g(x)=\mathbb{I}_{\mathcal{X}}(x)$ is the indicator function of a nonempty, closed convex set $\mathcal{X} \subset \mathbb{R}^d$, the proximal operator $\Pi_{g,h}$ reduces to projection onto $\mathcal{X}$. Due to projection, we must have $x_{k} \in \mathcal{X}$ for all $k$, implying $g(x_k)=0$. Therefore, the convergence result of Theorem \ref{thm: main result} holds for the suboptimality $f(x_{k})-f(x_{\star})$.
\end{remark}

\section{Further topics}
In this section, we consider further applications of the developed framework, namely, calculus of IQCs for various operators in optimization, continuous-time models and, more importantly, algorithm design. 
\subsection{Calculus of IQCs} \label{sec: Calculus of IQCs}
We now describe some operations on mappings from an IQC perspective, namely, inversion, affine operations, and function composition. These operations form a calculus that is useful for determining IQCs for commonly used nonlinear operators in optimization algorithms, such as proximal operators, projection operators, reflection operators, etc., and their compositions.

 It directly follows from the definition of pointwise IQCs in \eqref{eq: IQC} that if $\phi$ satisfies multiple pointwise IQCs defined by $(Q_{\phi,i},$ $x_{\star},$ $\phi(x_{\star}))$, $i=1,2,\ldots,\ell$, it also satisfies the pointwise IQC defined by $(\sum_{i=1}^{\ell} \! \sigma_i Q_{\phi,i}$ $,x_{\star},$ $\phi(x_{\star}))$, where $\sigma_i \geq 0,\ i=1,2,\ldots,\ell$. Further, $\phi$ also satisfies the IQC defined by $(Q,x_{\star},\phi(x_{\star}))$ for any $Q \succeq Q_{\phi}$. In the next two lemmas, we study the effect of inversion and affine transformation on IQCs.
\begin{lemma}[IQC for inversion] \label{eq: IQC inversion}
	Consider an invertible map $\phi \colon \mathbb{R}^d \to \mathbb{R}^d$ with $\phi^{-1}(\dom \, \phi) \subseteq \dom \, \phi$ satisfying the pointwise IQC defined by $(Q_{\phi},x_{\star},\phi(x_{\star}))$. Then, the inverse map $\phi^{-1} \colon \mathbb{R}^d \to \mathbb{R}^d$ satisfies the pointwise IQC defined by $(Q_{\phi^{-1}}$,$\phi(x_{\star})$,$x_{\star})$, where
	\begin{align} \label{eq: IQC inversion 1}
	Q_{\phi^{-1}} = \begin{bmatrix}
	0 & I_d \\ I_d & 0
	\end{bmatrix} Q_{\phi} \begin{bmatrix}
	0 & I_d \\ I_d & 0
	\end{bmatrix}.
	\end{align}
\end{lemma}
\begin{proof}
	By the substitution $x \leftarrow \phi^{-1}(x)$ in \eqref{eq: IQC}, we obtain
	\begin{align} \label{eq: IQC inversion 2}
	\begin{bmatrix}
	\phi^{-1}(x)-\phi^{-1}(x_{\star})\\ x-x_{\star}
	\end{bmatrix}^\top  Q_{\phi} \,     \begin{bmatrix}
	\phi^{-1}(x)-\phi^{-1}(x_{\star})\\ x-x_{\star}
	\end{bmatrix} \geq 0.
	\end{align}
	Further, we have 
	\begin{align} \label{eq: IQC inversion 3}
	\begin{bmatrix}
	\phi^{-1}(x)-\phi^{-1}(x_{\star})\\ x-x_{\star}
	\end{bmatrix} = \begin{bmatrix}
	0 & I_d \\ I_d & 0
	\end{bmatrix}     \begin{bmatrix}
	x-x_{\star}\\ \phi^{-1}(x)-\phi^{-1}(x_{\star})
	\end{bmatrix}.
	\end{align}
	Substituting \eqref{eq: IQC inversion 3} in \eqref{eq: IQC inversion 2} yields \eqref{eq: IQC inversion 1}.
\end{proof}
\begin{lemma}[IQC for affine operations] \label{eq: IQC affine transformation}
	Consider a map $\phi \colon \mathbb{R}^d \to \mathbb{R}^d$ satisfying the pointwise IQC defined by $(Q_{\phi},x_{\star},\phi(x_{\star}))$. Correspondingly, define the map $\psi(x)=S_2x + S_1 \phi(S_0x)$ with $S_0(\dom \, \phi) \subseteq \dom \, \phi$, where $S_0,S_1,S_2 \in \mathbb{R}^{d\times d}$, and $S_1$ is invertible. Then, $\psi$ satisfies the pointwise IQC defined by $(Q_{\psi},x_{\star},\psi(x_{\star}))$, where
	\begin{align} \label{eq: IQC affine transformation 1}
	Q_{\psi} = \begin{bmatrix}
	S_0^\top & -(S_1^{-1}S_2)^\top \\ 0 & S_1^{-1}
	\end{bmatrix} Q_{\phi} \begin{bmatrix}
	S_0 & 0 \\ -S_1^{-1}S_2 & (S_1^{-1})^{\top}
	\end{bmatrix}.
	\end{align}
\end{lemma}
\begin{proof}
	By the substitution $x \leftarrow S_0 x$ in \eqref{eq: IQC}, we obtain
	\begin{align} \label{eq: IQC affine transformation 2}
	\begin{bmatrix}
	S_0x-S_0x_{\star}\\ \phi(S_0x)-\phi(S_0x_{\star})
	\end{bmatrix}^\top  Q_{\phi} \,     \begin{bmatrix}
	S_0x-S_0x_{\star}\\ \phi(S_0x)-\phi(S_0x_{\star})
	\end{bmatrix}\geq 0.
	\end{align}
	Further, since $\psi(x)=S_2x + S_1 \phi(S_0x)$, we have 
	\begin{align} \label{eq: IQC affine transformation 3}
	\begin{bmatrix}
	S_0x-S_0x_{\star}\\ \phi(S_0x)-\phi(S_0x_{\star})
	\end{bmatrix} = \begin{bmatrix}
	S_0 & 0 \\ -S_1^{-1}S_2 & S_1^{-1}
	\end{bmatrix}     \begin{bmatrix}
	x-x_{\star}\\ \psi(x)-\psi(x_{\star})
	\end{bmatrix}.
	\end{align}
	Substituting \eqref{eq: IQC affine transformation 3} in \eqref{eq: IQC affine transformation 2} yields \eqref{eq: IQC affine transformation 1}.
\end{proof}

Finally, we study the composition of mappings. Specifically, consider the cascade connection of two mappings $\phi_1,\phi_2 \colon \mathbb{R}^d \to \mathbb{R}^d,\ i=1,2$ as in Figure \ref{fig: composition}, where $y = \phi_1(x)$ and $z=\phi_2(y)$. Further assume $\phi_1$ and $\phi_2$ satisfy pointwise IQCs defined by $(Q_{\phi_1},x_{\star},y_{\star})$ and $(Q_{\phi_2},y_{\star},z_{\star})$, respectively. By definition, these mappings impose the following quadratic constraints on the pairs $(x,y)$ and $(y,z)$:
\begin{align*}
\begin{bmatrix}
x-x_{\star} \\ y-y_{\star}
\end{bmatrix}^\top  Q_{\phi_1} \, \begin{bmatrix}
x-x_{\star} \\ y-y_{\star}
\end{bmatrix}  \geq 0, \quad \begin{bmatrix}
y-y_{\star} \\ z-z_{\star}
\end{bmatrix}^\top  Q_{\phi_2} \, \begin{bmatrix}
y-y_{\star} \\ z-z_{\star}
\end{bmatrix}  \geq 0.
\end{align*}
These two constraints separately define a quadratic constraint on the triple $(x,y,z)$, which can be encapsulated in a single constraint, as follows:
\begin{subequations} \label{eq: constraint on (x,y,z)}
	\begin{align} \label{eq: constraint on (x,y,z) 1}
	\begin{bmatrix}
	x - x_{\star} \\ y-y_{\star} \\ z-z_{\star}
	\end{bmatrix}^\top Q_{\psi} \begin{bmatrix}
	x - x_{\star} \\ y-y_{\star} \\ z-z_{\star}
	\end{bmatrix} \geq 0,
	\end{align}
	where $Q_{\psi} \in \mathbb{S}^{3d}$ is given by
	\begin{align} \label{eq: constraint on (x,y,z) 2}
	Q_{\psi} = \begin{bmatrix}
	I_d & 0 \\ 0 & I_d \\ 0 & 0
	\end{bmatrix} \sigma_1 Q_{\phi_1} \begin{bmatrix}
	I_d & 0 &0 \\ 0 & I_d & 0
	\end{bmatrix}  + \begin{bmatrix}
	0 & 0 \\ I_d & 0 \\ 0 & I_d
	\end{bmatrix}\sigma_2 Q_{\phi_2} \begin{bmatrix}
	0 & I_d &0 \\ 0 & 0 & I_d
	\end{bmatrix},
	\end{align}
\end{subequations}
with $\sigma_1,\sigma_2 \geq 0$. The quadratic constraint in \eqref{eq: constraint on (x,y,z) 1} follows by substituting \eqref{eq: constraint on (x,y,z) 2}  into \eqref{eq: constraint on (x,y,z) 1}. In the language of IQCs, we can say that the map $\psi = [\phi_1^\top \ (\phi_2 \circ \phi_1)^\top]^\top$ $\colon \mathbb{R}^d \to \mathbb{R}^{2d}$ satisfies the pointwise IQC defined by $(Q_{\psi},x_{\star},\psi(x_{\star}))$, where $Q_{\psi}$ is given by \eqref{eq: constraint on (x,y,z) 2}.
\begin{figure}[htbp]
	\centering
	\includegraphics[width=0.6\textwidth]{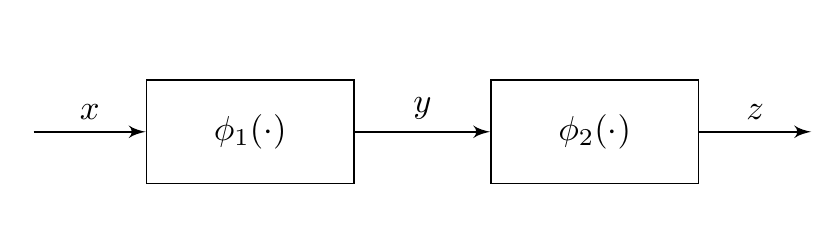}
	\caption{Cascade connection of two nonlinear mappings.}
	\label{fig: composition}
\end{figure}

We remark that the above treatment can be extended to multiple compositions. Specifically, for $\ell$ mappings in a cascade connection, the corresponding $\ell$ individual IQCs can be grouped into a single quadratic constraint on the concatenated vector of the input-output signals.

\subsubsection{Proximal operators} \label{subsection: IQCs for Proximal Operators}
Recall the definition of proximal operator for $f: \mathbb{R}^d \to \mathbb{R} \cup \{+\infty\}$:
\begin{align} \label{eq: proximal operator}
\Pi_{f,h}(x)=\mathrm{argmin}_{y \in \mathbb{R}^d} \{f(y) + \dfrac{1}{2h} \|y-x\|_2^2\}.
\end{align}
To characterize $\Pi_{f,h}$ from an IQC perspective, we note that for any given $x \in \dom \, f$, a necessary condition for optimality in \eqref{eq: proximal operator} is that
\begin{align}\label{eq: proximal operator optimality condition}
0 \in \partial f(\Pi_{g,h}(x)) + \dfrac{1}{h} (\Pi_{f,h}(x)-x), \ \mbox{for all }  x \in \dom \, f,
\end{align}
which is an implicit equation on $\Pi_{f,h}(x)$. In the next proposition, we show how to obtain a quadratic constraint for the proximal operator $\Pi_{f,h}$ from that of the subgradient $T_f$ by using the necessary optimality condition \eqref{eq: proximal operator optimality condition} that couples these two operators.

%
\begin{proposition}[IQCs for proximal operators] \label{prop: IQC For Proximal Operators}
	Let $f \colon \mathbb{R}^d \to \mathbb{R} \cup \{+\infty\}$ be a closed proper function whose subgradient $T_f$ satisfies the pointwise IQC defined by $(Q_f,x_{\star},T_f(x_{\star}))$, where $T_f(x_{\star}) \in \partial f(x_{\star})$. Then, the proximal operator $\Pi_{hf}$ satisfies the pointwise IQC defined by $(Q_{\Pi_{hf}},x_{\star},\Pi_{h f}(x_{\star}))$, where 
	\begin{align}\label{prop: IQC For Proximal Operators 1}
	Q_{\Pi_{hf}} = \begin{bmatrix}
	0  &  h^{-1}I_d \\ I_d & -h^{-1}I_d
	\end{bmatrix}Q_{f} \begin{bmatrix}
	0  &  I_d \\ h^{-1}I_d & -h^{-1}I_d
	\end{bmatrix}.
	\end{align}
\end{proposition}
\begin{proof}
	Suppose $T_f(x) \in \partial f(x)$ ($T_f(x)=\nabla f(x)$ when $f$ is differentiable) satisfies the pointwise IQC defined by $(Q_f,x_{\star},T_f(x_{\star}))$. 
	By the substitution $x \leftarrow \Pi_{hf}(x)$ and $x_{\star} \leftarrow \Pi_{hf}(x_{\star})$ in \eqref{eq: IQC}, we obtain
	\begin{align} \label{prop: IQC For Proximal Operators 2}
	\begin{bmatrix}
	\Pi_{hf}(x)-\Pi_{hf}(x_{\star}) \\ T_f(\Pi_{hf}(x))-T_f(\Pi_{hf}(x_{\star}))
	\end{bmatrix}^\top Q_f \, \begin{bmatrix}
	\Pi_{hf}(x)-\Pi_{hf}(x_{\star}) \\ T_f(\Pi_{hf}(x))-T_f(\Pi_{hf}(x_{\star}))
	\end{bmatrix} \geq 0.
	\end{align}
	On the other hand, by the optimality condition \eqref{eq: proximal operator optimality condition}, we have $T_f(\Pi_{h f}(x))=\frac{1}{h}(x-\Pi_{h f}(x))$. Substituting this into \eqref{prop: IQC For Proximal Operators 2}, we obtain
	\begin{align} \label{prop: IQC For Proximal Operators 3}
	\begin{bmatrix}
	\Pi_{hf}(x)-\Pi_{hf}(x_{\star}) \\     \dfrac{1}{h} (x - x_{\star})-\dfrac{1}{h} (\Pi_{hf}(x)- \Pi_{hf}(x_{\star})) 
	\end{bmatrix}^\top \! Q_f \! \begin{bmatrix}
	\Pi_{hf}(x)-\Pi_{hf}(x_{\star}) \\ \dfrac{1}{h} (x - x_{\star})-\dfrac{1}{h} (\Pi_{hf}(x)- \Pi_{hf}(x_{\star})) 
	\end{bmatrix} \geq 0.
	\end{align}
	Further, we can write
	\begin{align} \label{prop: IQC For Proximal Operators 4}
	\begin{bmatrix}
	\Pi_{hf}(x)-\Pi_{hf}(x_{\star}) \\     \dfrac{1}{h} (x - x_{\star})-\dfrac{1}{h} (\Pi_{hf}(x)- \Pi_{hf}(x_{\star})) 
	\end{bmatrix} = \begin{bmatrix}
	0  &  I_d \\ \dfrac{1}{h}I_d & -\dfrac{1}{h}I_d
	\end{bmatrix} \begin{bmatrix}
	x-x_{\star} \\ \Pi_{hf}(x)-\Pi_{hf}(x_{\star})
	\end{bmatrix}.
	\end{align}
	By substituting \eqref{prop: IQC For Proximal Operators 4} in \eqref{prop: IQC For Proximal Operators 3}, we will arrive at the desired inequality in \eqref{prop: IQC For Proximal Operators 1}.
\end{proof}
Notice that by \eqref{eq: proximal operator optimality condition}, we have that $\Pi_{h f} = (I+h\partial f)^{-1}$. In other words, the proximal operator is obtained by the operations $\partial f \to I+h \partial f \to (I+h \partial f)^{-1}$, i.e.,  an affine operation on $\partial f$ followed by an inversion. Therefore, for obtaining the IQC of $\Pi_{h f}$ from that of $\partial f$, we can directly use Lemma \ref{eq: IQC inversion} and \ref{eq: IQC affine transformation} to arrive at an alternative derivation of \eqref{prop: IQC For Proximal Operators 1}. 
\subsubsection{IQCs for projection operators} \label{subsection: IQC for projection operators}The projection operator is the proximal operator $\Pi_{h f}$ for the particular selection $f(x)=\mathbb{I}_{\mathcal{X}}(x)$, where $\mathbb{I}_{\mathcal{X}}$ is the extended-value indicator function of the nonempty closed convex set $\mathcal{X} \subset \mathbb{R}^d$ onto which we project. Since $f$ is nondifferentiable and convex in this case, its subgradient operator $T_f$ satisfies the pointwise IQC defined by $(Q_f,x_{\star},T_f(x_{\star}))$, where $Q_f$ is given by \eqref{eq: strongly convex IQC} with $L_f=\infty$. It then follows from Proposition \ref{prop: IQC For Proximal Operators} that the projection operator $\Pi_{\mathcal{X}}$ satisfies the IQC defined by $(Q_{\Pi_{\mathcal{X}}},x_{\star},\Pi_{\mathcal{X}}(x_{\star}))$, where
\begin{align}
Q_{\Pi_{\mathcal{X}}} = \begin{bmatrix}
0 & \dfrac{1}{2} \\ \dfrac{1}{2} & -1
\end{bmatrix} \otimes I_d.
\end{align}
This IQC corresponds to the firm nonexpansiveness property of the projection operator \cite{combettes2011proximal}, which implies the Lipschitz continuity of $\Pi_{\mathcal{X}}$ with Lipschitz parameter equal to one. 

%
%

\subsection{Beyond convexity} 
The convergence analysis of several algorithms do not make a full use of convexity. In other words, convexity is sufficient for convergence of these algorithms. This has motivated the introduction of function classes that are relaxation of convexity. In this subsection, we briefly discuss some of these classes and how they can be related to the developed framework in this paper. Formally, consider a continuously differentiable function $f \colon$ $\mathbb{R}^d\to \mathbb{R}$ that satisfies the following bounds.
\begin{align} \label{eq: beyond convexity}
\begin{bmatrix}
x-x_{\star} \\ \nabla f(x)
\end{bmatrix} ^\top {R}'_{f} \begin{bmatrix}
x-x_{\star} \\ \nabla f(x)
\end{bmatrix} \leq f(x)-f(x_{\star}) \leq 	\begin{bmatrix}
x-x_{\star} \\ \nabla f(x)
\end{bmatrix} ^\top R_{f} \begin{bmatrix}
x-x_{\star} \\ \nabla f(x)
\end{bmatrix},
\end{align}
where $R_f,R'_f \in \mathbb{S}^{2d}$ are symmetric matrices and $x_{\star}$ is such that $\nabla f(x_{\star})=0$. It follows from \eqref{eq: beyond convexity} that 
\begin{align} \label{eq: beyond convexity IQC}
\begin{bmatrix}
x-x_{\star} \\ \nabla f(x)
\end{bmatrix}^\top (R_f-R'_f) \begin{bmatrix}
x-x_{\star} \\ \nabla f(x)
\end{bmatrix} \geq 0.
\end{align}
Note that since $\nabla f(x_{\star})=0$, the above inequality implies that $\nabla f$ satisfies the pointwise IQC defined by $(R_f-R'_f,x_{\star},\nabla f(x_{\star}))$. 
Several function classes can be written in the form \eqref{eq: beyond convexity}, where $R_f$ and $R'_f$ differ for each class. We give three examples below.

\medskip

\textit{(Strongly) convex functions.} In $\S$\ref{sec: IQCs for (strongly) convex functions}, we considered IQCs for convex functions. Specifically, the quadratic inequality \eqref{eq: strongly convex lipschitz 2} is necessary and sufficient for the inclusion $f \in \mathcal{F}(m_f,L_f)$. An equivalent inequality involving function values is \cite{taylor2017smooth}\footnote{Note that, by adding both sides of \eqref{eq: convex interpolation} to the inequality obtained by interchanging $x$ and $y$ in \eqref{eq: convex interpolation}, we obtain \eqref{eq: strongly convex lipschitz 2}.} 
\begin{align} \label{eq: convex interpolation}
&f(y)\!-\!f(x)\!-\!\nabla f(x)^\top(y\!-\!x)	\geq \frac{1}{2(L_f\!-\!m_f)} \|\nabla f(y)\!-\!\nabla f(x)\|_2^2 \\ & \qquad \qquad \qquad \quad+\frac{m_fL_f}{2(L_f\!-\!m_f)} \|y\!-\!x\|_2^2\!-\!\dfrac{m_f}{L_f\!-\!m_f}(\nabla f(y)\!-\!\nabla f(x))^\top (y-x), \nonumber
\end{align}
If we restrict \eqref{eq: convex interpolation} to hold only for the particular selection $(x,y)=(x_\star,x)$ and $(x,y)=(x,x_{\star})$, we obtain a new class of functions that can be put in the form \eqref{eq: beyond convexity} with $R'_f, R_f $ given by
\begin{align} \label{eq: nonstrongly convex class 1}
{R}'_f &= \begin{bmatrix}
\frac{m_fL_f}{2(L_f-m_f)} & \frac{-m_f}{2(L_f-m_f)}\\  \frac{-m_f}{2(L_f-m_f)} & \frac{1}{2(L_f-m_f)}
\end{bmatrix}\otimes I_d, \quad {R}_f = \begin{bmatrix}
\frac{-m_fL_f}{2(L_f-m_f)} & \frac{L_f}{2(L_f-m_f)}\\  \frac{L_f}{2(L_f-m_f)} & \frac{-1}{2(L_f-m_f)}
\end{bmatrix} \otimes I_d,
\end{align}
Using \eqref{eq: beyond convexity IQC}, we can conclude
\begin{align}
\begin{bmatrix}
x-x_{\star} \\ \nabla f(x)
\end{bmatrix}^\top \begin{bmatrix} -\frac{m_f L_f}{m_f+L_f} I_d & \frac{1}{2} I_d \\ \frac{1}{2} I_d & -\frac{1}{m_f+L_f} I_d \end{bmatrix} \begin{bmatrix}
x-x_{\star} \\ \nabla f(x)
\end{bmatrix} \geq 0.
\end{align}
Note that this quadratic inequality is the same as that of convex functions but only holds when the reference point $x_{\star}$ in the definition of pointwise IQC satisfies $\nabla f(x_{\star})=0$.

\medskip
\textit{Weakly smooth weakly quasiconvex functions.} Suppose $f$ is continuously differentiable and satisfies \cite{hardt2016gradient}:
%
\begin{align} \label{eq: QWC}
\dfrac{1}{\Gamma_f} \|\nabla f(x)\|_2^2\leq f(x)-f(x_{\star}) \leq \dfrac{1}{\tau_f}\nabla f(x)^\top (x-x_{\star}) \quad \text{for all} \ x \in \mathcal{S},
\end{align}
where $x_{\star}$ is a global minimum of $f$, and $0 < \tau_f,\Gamma_f < \infty$. These inequalities ensure that any point with vanishing gradient is optimal \cite{hardt2016gradient}, i.e., $\nabla f(x_{\star})=0$. The inequality \eqref{eq: QWC} can be put in the form \eqref{eq: beyond convexity}, where $R'_f, R_f,$ and $Q_f$ are given by
\begin{align} \label{eq: nonstrongly convex class 2}
{R}'_f = \begin{bmatrix}
0 & 0 \\  0 & \frac{1}{\Gamma_f}
\end{bmatrix}\otimes I_d, \ \ 	{R}_f = \begin{bmatrix}
0 & \frac{1}{2\tau_f}\\  \frac{1}{2\tau_f} & 0
\end{bmatrix} \otimes I_d, \ \ Q_f = \begin{bmatrix}
0  &  \frac{1}{2\tau_f} \\ \frac{1}{2\tau_f} & -\frac{1}{\Gamma_f}
\end{bmatrix}  \otimes I_d.
\end{align}

\medskip

\emph{Polyak-$\L{}$ojasiewicz (PL) condition.} Suppose $f$ is continuously differentiable and satisfies 
        \begin{align} \label{eq: PL 1}
        0 \leq f(x) - f(x_{\star}) \leq \dfrac{1}{2m_f} \|\nabla f(x)\|_2^2 \quad \text{for all} \ x \in \mathcal{S},
        \end{align}
        for some $m_f>0$. Again, this class can be put in the form \eqref{eq: beyond convexity}. 
%

\subsection{Continuous-time models} \label{Section: Continuous-Time Models}
There is a close connection between iterative algorithms and discretization of ordinary differential equations (ODE). In fact, many iterative first-order optimization algorithms reduce to their ``generative" ODEs by time-scaling and infinitesimal step sizes. 
In this subsection, we consider convergence analysis of continuous-time models for solving the unconstrained problem in \eqref{eq: unconstrained optimization}. Specifically, consider the following continuous-time dynamical system in state-space form:
\begin{align} \label{eq: continuous time dynamical system gradient}
\dot{\xi}(t) = A(t) \xi(t)+ B(t) u(t), \ \  y(t) = C(t)\xi(t), \ \ u(t) = \nabla f(y(t)) \ \ \mbox{for all } t\geq t_0,
\end{align}
where at each continuous time $t \geq t_0$, $\xi(t) \in \mathbb{R}^n$ is the state, $y(t) \in \mathbb{R}^d$ is the output ($d \leq n$), and $u(t)= \nabla f(y(t))$ is the feedback input. We assume \eqref{eq: continuous time dynamical system gradient} solves \eqref{eq: unconstrained optimization} asymptotically from all admissible initial conditions, i.e., $y(t)$ satisfies $\lim_{t \to \infty} f(y(t)) = f(y_\star)$, where the optimal point $y_\star$ obeys $\nabla f(y_{\star})=0$. Therefore, any fixed point of \eqref{eq: continuous time dynamical system gradient} satisfies
\begin{align} \label{eq: continuous time dynamics fixed points}
0 = A(t) \xi_{\star}, \quad y_\star = C(t) \xi_{\star}, \quad u_{\star}  = \nabla f(y_\star) = 0 \quad \mbox{for all } t\geq t_0.
\end{align}
We replicate the convergence analysis of discrete-time models using the Lyapunov function
\begin{align} \label{eq: Lyapunov gradient cont}
V(\xi(t),t) = a(t) (f(y(t))-f(y_\star)) + (\xi(t)-\xi_\star)^\top P(t)  (\xi(t)-\xi_\star),
\end{align}
where $(\xi(t),y(t))$ satisfies \eqref{eq: continuous time dynamical system gradient} and $(\xi_{\star},y_{\star})$ satisfies \eqref{eq: continuous time dynamics fixed points}. The Lyapunov function is parameterized by $P(t) \in \mathbb{S}_{+}^n$, as well as $a(t) \geq 0$.
If $a(t)$ and $P(t)$ are such that $\dot{V}(\xi(t),t) \leq 0$, then we could guarantee that $V(\xi(t),t) \leq V(\xi(t_0),t_0)$, which in turn implies
\begin{align} \label{eq: convergence rate continuous}
0 \leq f(y(t))-f(y_{\star})\leq V(\xi(t_0),t_0)/a(t) = \mathcal{O}(1/a(t)) \quad \mbox{for all } t\geq t_0.
\end{align}
In other words, $a(t)$ provides a lower bound on the convergence rate. Ideally, we are interested in finding the best bound, which translates into the fastest growing $a(t)$. In the following theorem, we develop an LMI to find such an $a(t)$.
%
%
%
\begin{theorem}\label{thm: Continuous-Time Gradient Descent}

Let $f \in \mathcal{F}(m_f,L_f)$ and consider the continuous-time dynamics in \eqref{eq: continuous time dynamical system gradient}, whose fixed points satisfy \eqref{eq: continuous time dynamics fixed points}. Suppose there exist a differentiable nondecreasing $a(t) \colon [t_0,\infty) \to \mathbb{R}_{+}$, a differentiable $P(t) \colon [t_0,\infty) \to \mathbb{S}_{+}^n$, and a continuous $\sigma(t) \colon [t_0,\infty) \to \mathbb{R}_{+}$ that satisfy
		\begin{align} \label{thm: Continuous-Time Gradient Descent 11}
		M_0(t) + a(t) M_1(t) + \dot{a}(t) M_2(t)+ \sigma(t) M_3(t)\preceq 0 \quad \mbox{for all } t \geq t_0,
		\end{align}
		where
		\begin{align*}
		M_0(t) &= \begin{bmatrix} P(t)A(t)\!+\!A(t)^\top P(t)\!+\!\dot{P}(t) & P(t)B(t) \\  B(t)^\top P(t)& 0 \end{bmatrix}, \\
		M_1(t) &= \dfrac{1}{2}\!\begin{bmatrix} 0 & (C(t)A(t)+\dot{C}(t))^\top \\  C(t)A(t)+\dot{C}(t)&  C(t)B(t)+B(t)^\top C(t)^\top \end{bmatrix}, \nonumber \\
		M_2(t) & =  \begin{bmatrix} C(t)^\top & 0 \\ 0 & I_d \end{bmatrix} \begin{bmatrix}
			-\frac{m_f}{2} I_d & \frac{1}{2} I_d \\ \frac{1}{2} I_d & 0
		\end{bmatrix}\begin{bmatrix}C(t) & 0 \\ 0 & I_d \end{bmatrix}, \nonumber \\
		M_3(t) & = \begin{bmatrix} C(t)^\top & 0 \\ 0 & I_d \end{bmatrix} \begin{bmatrix}
			-\frac{m_f L_f}{m_f + L_f}I_d & \frac{1}{2} I_d \\ \frac{1}{2} I_d & -\frac{1}{m_f+L_f}I_d
		\end{bmatrix}\begin{bmatrix}C(t) & 0 \\ 0 & I_d \end{bmatrix}, \nonumber
		\end{align*}
	%
	Then, for any $y(t_0) \in \dom \, f$, the following inequality holds for all $t \geq t_0$.
	\begin{align} \label{thm: Continuous-Time Gradient Descent 2}
	f(y(t))-f(y_{\star}) \leq \dfrac{a(t_0)(f(y(t_0))\!-\!f(y_{\star}))+(\xi(t_0)\!-\!\xi_\star)^\top P(t_0)  (\xi(t_0)\!-\!\xi_\star)}{a(t)}
	\end{align}
\end{theorem}
\begin{proof}
	It suffices to show that the LMI condition in \eqref{thm: Continuous-Time Gradient Descent 11} implies $\dot{V}(\xi(t),t) \leq 0$.  The time derivative of the Lyapunov function \eqref{eq: Lyapunov gradient cont} is
	\begin{align} \label{thm: Continuous-Time Gradient Descent 3}
	\dot{V} = \dot{a} &(f(y)-f(y_{\star}))+a\nabla f(y)^\top \dot{y}\! +\!2(\xi-\xi_{\star})^\top P \dot \xi+(\xi-\xi_{\star})^\top \dot{P} (\xi-\xi_{\star}). 
	\end{align}
	We have dropped the arguments for notational simplicity. We proceed to bound all the terms in the right-hand side of \eqref{thm: Continuous-Time Gradient Descent 3}, using the assumption $f \in \mathcal{F}(m_f,L_f)$. By invoking (strong) convexity, we can write
	\begin{align} \label{thm: Continuous-Time Gradient Descent 4}
	f(y) \!-\! f(y_\star) &\leq  \begin{bmatrix}
	y - y_\star \\ \nabla f(y) - \nabla f(y_\star)
	\end{bmatrix}^\top \begin{bmatrix}
	-\frac{m_f}{2} I_d & \frac{1}{2}I_d \\ \frac{1}{2}I_d & 0
	\end{bmatrix}  \, \begin{bmatrix}
	y - y_\star \\ \nabla f(y) - \nabla f(y_\star)
	\end{bmatrix}  
	\\
	%
	%
	& = 
	\begin{bmatrix}
	\xi-\xi_{\star} \\ u-u_{\star}
	\end{bmatrix}^\top \begin{bmatrix}
	C& 0 \\ 0 & I_d
	\end{bmatrix}^\top \begin{bmatrix}
	-\frac{m_f}{2} I_d & \frac{1}{2}I_d \\ \frac{1}{2}I_d & 0
	\end{bmatrix} \begin{bmatrix}
	C & 0 \\ 0 & I_d
	\end{bmatrix} \,  \begin{bmatrix}
	\xi-\xi_{\star} \\ u-u_{\star}
	\end{bmatrix}. \nonumber \\
	&= e^\top M_2 e. \nonumber 
	\end{align}
	where we have defined $e=\begin{bmatrix} (\xi-\xi_{\star})^\top & (u-u_{\star})^\top \end{bmatrix}$. Further, we can write
	\begin{align} \label{thm: Continuous-Time Gradient Descent 44}
	&\nabla f(y)^\top \dot{y}= (u-u_{\star})^\top (CA(\xi-\xi_{\star}) + CB(u-u_{\star})+\dot{C}(\xi-\xi_{\star}))  \\
	&= \begin{bmatrix}
	\xi -\xi_{\star} \\ u -u_{\star}
	\end{bmatrix}^\top \begin{bmatrix}
	0 & \frac{1}{2} (C A +\dot C )^\top \\  \frac{1}{2}(C A +\dot C )&  \frac{1}{2} (C B +B ^\top C ^\top)
	\end{bmatrix} \begin{bmatrix}
	\xi -\xi_{\star} \\ u -u_{\star}
	\end{bmatrix}\nonumber \\
	& = e ^\top M_1  e . \nonumber
 	\end{align}
	where we have used \eqref{eq: continuous time dynamical system gradient} and \eqref{eq: continuous time dynamics fixed points}. Similarly, we can write
	\begin{align} \label{thm: Continuous-Time Gradient Descent 444}
	2(\xi -\xi_{\star})^\top P  \dot \xi  &= \begin{bmatrix}
	\xi -\xi_{\star} \\ u -u_{\star}
	\end{bmatrix}^\top
	\begin{bmatrix}
	P A +A ^\top P  & P B  \\  B ^\top P ^\top& 0
	\end{bmatrix} \begin{bmatrix}
	\xi -\xi_{\star} \\ u -u_{\star}
	\end{bmatrix}= e ^\top M_0  e . 
	\end{align}
	Finally, since $f \in \mathcal{F}(m_f,L_f)$, $\nabla f$ satisfies the quadratic constraint in \eqref{eq: strongly convex IQC}. Therefore, we can write
	\begin{align} \label{thm: Continuous-Time Gradient Descent 5}
	e ^\top M_3  e  &=\begin{bmatrix}
	\xi -\xi_{\star} \\ u -u_{\star}
	\end{bmatrix}^\top \begin{bmatrix} C  & 0 \\ 0 & I_d \end{bmatrix}^\top \begin{bmatrix}
	-\frac{m_f L_f}{m_f + L_f}I_d & \frac{1}{2} I_d \\ \frac{1}{2} I_d & -\frac{1}{m_f+L_f}I_d
\end{bmatrix} \begin{bmatrix}C  & 0 \\ 0 & I_d \end{bmatrix}  \, \begin{bmatrix}
	\xi -\xi_{\star} \\ u -u_{\star}
	\end{bmatrix} \\ &= \!\begin{bmatrix}
	y  - y_\star \\ u  - u_{\star}
	\end{bmatrix}^\top \begin{bmatrix}
	-\frac{m_f L_f}{m_f + L_f}I_d & \frac{1}{2} I_d \\ \frac{1}{2} I_d & -\frac{1}{m_f+L_f}I_d
\end{bmatrix} \begin{bmatrix}
	y  - y_\star \\ u  - u_{\star}
	\end{bmatrix} \geq 0. \nonumber 
	\end{align}
	By substituting \eqref{thm: Continuous-Time Gradient Descent 4}-\eqref{thm: Continuous-Time Gradient Descent 444} in \eqref{thm: Continuous-Time Gradient Descent 3} and rearranging terms, we can write
	\begin{align}  \label{thm: Continuous-Time Gradient Descent 6}
	\dot{V} & \leq e ^\top \left(M_0 + a M_1 + \dot{a} M_2  \right) e  
	\end{align}
	The LMI in \eqref{thm: Continuous-Time Gradient Descent 11} implies
	\begin{align}\label{thm: Continuous-Time Gradient Descent 7}
	M_0 + a M_1 + \dot{a} M_2  \preceq -\sigma  M_3 
	\end{align}
	Multiplying \eqref{thm: Continuous-Time Gradient Descent 7} on the left and right by $e ^\top$ and $e $, respectively, and substituting the result back in \eqref{thm: Continuous-Time Gradient Descent 6} yields
	\begin{align*}
	\dot{V} & \leq -\sigma  e ^\top M_3  e  \leq 0,
	\end{align*}
	where the second inequality follows from \eqref{thm: Continuous-Time Gradient Descent 5}. The proof is now complete.
\end{proof}

According to Theorem \ref{thm: Continuous-Time Gradient Descent}, we can find the rate generating function $a(t)$ by solving the LMI in \eqref{thm: Continuous-Time Gradient Descent 11}. More precisely, this LMI defines a first-order differential inequality on $a(t)$ whose solutions certify an $\mathcal{O}(1/a(t))$ convergence rate. 
The best lower bound on the convergence rate (i.e., the fastest growing $a(t)$) can be found by solving the following symbolic optimization problem:
\begin{align} \label{eq: SDP gradient cont}
\underset{\dot{a}(t) \geq 0,\sigma(t) \geq 0}{\mbox{maximize}} \ \dot{a}(t) \ \mbox{subject to} \ \dot{a}(t) M_0(t) + a(t) M_1(t) + M_2(t)+\sigma(t) M_3(t) \preceq 0,
\end{align} 
The optimality condition for \eqref{eq: SDP gradient cont} translates into a first-order differential equation (ODE) on $a(t)$. The solution to this ODE yields the best rate bound that can be certified using the Lyapunov function \eqref{eq: Lyapunov gradient cont}.
In the following, we specialize the model in \eqref{eq: continuous time dynamical system gradient} to the particular case of the gradient flow ($\S$\ref{sec: Continuous-time gradient flow}) and its accelerated variant ($\S$\ref{sec: Continuous-time accelerated gradient flow}), where we will use Theorem \ref{thm: Continuous-Time Gradient Descent} to derive the corresponding convergence rates.
%

%

\subsubsection{Continuous-time gradient flow}  \label{sec: Continuous-time gradient flow}
Consider the following ODE for solving \eqref{eq: unconstrained optimization}:
\begin{align} \label{eq: gradient flow}
\dot{x}(t) = - \alpha \nabla f(x(t)), \quad x(0) \in \dom \, f,
\end{align}
where $\alpha>0$. This ODE can be represented in the form of \eqref{eq: continuous time dynamical system gradient} with $n=d$, and $(A,B,C)=(0_d,-\alpha I_d,I_d)$. By selecting $P(t)=p  I_d, \ p \geq 0,$ and applying the dimensionality reduction outlined in Remark \ref{remark: structure}, we obtain the following LMI:
\begin{align}
\begin{bmatrix}
-\dfrac{m_f}{2}\dot{a}(t) & \frac{1}{2}\dot{a}(t)-p  \alpha  \\  \frac{1}{2}\dot{a}(t)-p  \alpha  & -\alpha a(t) 
\end{bmatrix} + \sigma(t) \begin{bmatrix}
\frac{-m_fL_f}{m_f+L_f} & \frac{1}{2} \\ \frac{1}{2}  & \frac{-1}{m_f+L_f} 
\end{bmatrix}\preceq 0.
\end{align}
By elementary calculations, it can be verified that the solution to the corresponding optimization problem in \eqref{eq: SDP gradient cont} is $\sigma(t) = 0$, and $\dot{a}(t) = 2p  + m_f\alpha a(t) + ((m_f\alpha a(t))^2+2p  m_f \alpha a(t))^{1/2}$. Setting $p =0$ and solving the latter ODE with initial condition $a(0) > 0$ yields $a(t) = a(0)\exp(2m_f \alpha t)$. Therefore, the gradient flow \eqref{eq: gradient flow} exhibits the following convergence rate for strongly convex $f$:
\begin{align*}
f(x(t)) - f(x_{\star}) \leq e^{-2m_f \alpha t}(f(x(0))-f(x_{\star})).
\end{align*}
Now we consider convex functions ($m_f=0$) for which the LMI reduces to
\begin{align*}
\begin{bmatrix}
0 & \frac{1}{2}\dot{a}(t)-p  \alpha + \dfrac{\sigma(t)}{2} \\  \frac{1}{2}\dot{a}(t)-p  \alpha + \dfrac{\sigma(t)}{2} & -\alpha a(t)- \dfrac{\sigma(t)}{L_f}
\end{bmatrix} \leq 0.
\end{align*}
This LMI condition is equivalent to the condition $\dot{a}(t) \leq 2p  \alpha -\sigma(t)$. Therefore, by setting $\sigma(t)=0$, we obtain the optimal (fastest growing) $a(t)$, which satisfies the ODE $\dot{a}(t)=2p  \alpha$. Solving this ODE with the initial condition $a(0) \geq 0$, we obtain the following rate bound.
\begin{align*} 
f(x(t)) - f(x_{\star}) \leq \frac{a(0) (f(x(0))-f(x_{\star}))+p  \|x(0)-x_{\star}\|_2^2}{a(0)+2p  \alpha t}.
\end{align*}
%
%
%
\subsubsection{Continuous-time accelerated gradient flow} \label{sec: Continuous-time accelerated gradient flow}
As a second case study, we consider the following second-order ODE for solving \eqref{eq: unconstrained optimization}:
\begin{align} \label{eq: Continuous-time accelerated gradient flow}
\ddot{x}(t) + \dfrac{r}{t} \dot{x}(t) + \nabla f(x(t))=0,\ r>0.
\end{align}
This ODE is the continuous-time limit of Nesterov's accelerated scheme combined with an appropriate time scaling \cite{su2016differential}. The ODE \eqref{eq: Continuous-time accelerated gradient flow} and its variants have been investigated extensively in the literature \cite{alvarez2000minimizing,cabot2009long,attouch2016fast}. A state-space representation of \eqref{eq: Continuous-time accelerated gradient flow} is given by
%
\begin{align} \label{eq: Continuous-time accelerated gradient flow 0}
\dot{\xi}(t) &= \begin{bmatrix}
-\dfrac{r-1}{t}I_d& \dfrac{r-1}{t}I_d \\ 0 & 0
\end{bmatrix} \xi(t) + \begin{bmatrix}
0 \\ -\dfrac{t}{r-1}I_d
\end{bmatrix} \nabla f(y(t)), \\
y(t) &= \begin{bmatrix}
I_d & 0
\end{bmatrix}\xi(t), \nonumber 
\end{align}
%
where $\xi_1 = x$, $\xi_2=x + t/(r-1)\dot{x}$ are the states, $\xi = [\xi_1^\top \quad \xi_2^\top]^\top \in \mathbb{R}^{2d}$ is state vector, and $y=\xi_1$ is the output. The fixed points of \eqref{eq: Continuous-time accelerated gradient flow 0} are  $(\xi_{\star},y_{\star},u_{\star})=([x_{\star}^\top \ x_{\star}^\top]^\top,x_{\star},0)$, where $x_{\star} \in \mathcal{X}_{\star}$ is any optimal solution satisfying $\nabla f(x_{\star})=0$.
%
%

We now analyze the convergence rate of \eqref{eq: Continuous-time accelerated gradient flow 0} for convex functions ($m_f=0$). By selecting $P(t)=\hat{P} I_d$, where $\hat{P} \in \mathbb{S}^2_{++}$ is time-invariant, and applying the dimensionality reduction of Remark \ref{remark: structure}, we arrive at the following $3 \times 3$ LMI,
%

\begin{align*}
\begin{bmatrix}
-\frac{2(r-1)p_{11}}{t} & \frac{(r-1)(p_{11}-p_{21})}{t}& \frac{\dot{a}(t)+\sigma}{2}-\frac{(r-1)a(t)}{2t}-\frac{tp_{12}}{r-1}\\ \frac{(r-1)(p_{11}-p_{21})}{t} & \frac{2(r-1)p_{21}}{t} & \frac{(r-1)a(t)}{2t}-\frac{tp_{22}}{r-1}\\ \frac{\dot{a}(t)+\sigma}{2}-\frac{(r-1)a(t)}{2t}-\frac{t p_{12}}{r-1}& \frac{(r-1)a(t)}{2t}-\frac{tp_{22}}{r-1} & - \frac{\dot{a}(t)}{2L_f} - \frac{\sigma}{L_f}
\end{bmatrix} \preceq 0,
\end{align*}
where $\hat{P}=[p_{ij}]$. A simple analytic solution to the above LMI can be obtained by choosing $p_{11}=p_{12}=p_{21}=0$. With this particular choice, the LMI simplifies to the following conditions:
\begin{align}
\frac{\dot{a}(t)+\sigma(t)}{2}-\frac{(r-1)a(t)}{2t} = 0, \quad p_{22}=(\frac{r-1}{t})^2\frac{a(t)}{2}.
\end{align}
Using the assumption that $p_{22}$ is constant together with the condition $\sigma(t) \geq 0$, the above conditions enforce $a(t) = c t^2$, and $p_{22}=c{(r-1)^2}/2$ for arbitrary $c>0$ along with the condition $r \geq 3$. Using Theorem \ref{thm: Continuous-Time Gradient Descent}, we obtain the convergence rate:
\begin{align*}
f(x(t))-f(x_{\star}) \leq \dfrac{(r-1)^2 \|x(0)-x_{\star}\|_2^2}{2t^2} \quad   r \geq 3.
\end{align*}
This convergence result agrees with \cite[Theorem 5]{su2016differential}. More generally, by allowing the matrix $P(t)$ to be time-dependent, the LMI \eqref{thm: Continuous-Time Gradient Descent 11} can be used to directly answer the following question: How does the convergence rate of the accelerated gradient flow change with the parameter $r$.

\subsection{Algorithm design} 
%
In this subsection, we briefly explore algorithm tuning and design using the developed LMI framework. In particular, we consider robustness as a design criterion.
%
It has been shown in \cite{devolder2014first,lessard2016analysis,cyrus2017robust} that there is a trade-off between an algorithm's rate of convergence and its robustness against inexact information about the oracle. In particular, fast methods such as the Nesterov's accelerated method require first-order information with higher accuracy than standard gradient methods to obtain a solution with a given accuracy \cite{devolder2014first}. To explain this trade-off in our framework, we recall the proof of Theorem \ref{thm: main result}, in which we showed that the following LMI 
\begin{align}
M_k^{0}  + a_{k} M_k^{1} + (a_{k+1}-a_k) M_k^{2} + \sigma_k M_k^{3}\preceq 0 \quad \text{for all} \  k,
\end{align}
ensures that the Lyapunov function satisfies
\begin{align} \label{eq: Lyapunov function decrease margin}
V_k(\xi_{k+1}) \leq V(\xi_{k}) -\sigma_k e_k^\top M_k^3 e_k \quad \mbox{for all} \ k.
\end{align}
In view of \eqref{eq: Lyapunov function decrease margin}, the nonnegative term $\sigma_k e_k^\top M_k^3 e_k$ provides an additional stability margin and hence, safeguards the algorithm against uncertainties in the algorithm or underlying assumptions. Based on this observation, we propose the LMI
\begin{align}\label{eq: LMI robust}
M_k^{0}  + a_{k} M_k^{1} + (a_{k+1}-a_k) M_k^{2} + \sigma_k M_k^{3} + S_k \preceq 0 \quad \text{for all} \  k,
\end{align}
where $S_k$ is any symmetric matrix that satisfies $e_k^\top S_k e_k \geq 0$ for all $k$. In particular, any $S_k \succeq 0$ is a valid choice. By revisiting the proof of Theorem \ref{thm: main result}, the feasibility of the above LMI imposes the stricter condition 
\begin{align} \label{eq: robust Lyapunov decrease}
V_{k+1}(\xi_{k+1}) \leq V_k(\xi_{k}) - e_k^\top (\sigma_k M_k^3 + S_k) e_k \quad e_k^\top S_k e_k \geq 0,
\end{align}
on the decrement of the Lyapunov function. The LMI in \eqref{eq: LMI robust} is the robust counterpart of \eqref{thm: main 1}. Now we can use \eqref{eq: LMI robust} to search for the parameters of the algorithm, considering $S_k$ as a tuning parameter that makes the trade-off between robustness and rate of convergence. 

\medskip

\emph{Robust gradient method.} As an illustrative example, consider the gradient method applied to $f \in \mathcal{F}(m_f,L_f)$. Consider the robust counterpart of the LMI in \eqref{eq: gradient method SC LMI}:
\begin{align} 
\begin{bmatrix}
p\!-\! \rho^2 p& -h p\\  -h p& h^2 p
\end{bmatrix} + \lambda \begin{bmatrix}
\frac{-m_fL_f}{m_f+L_f} & \frac{1}{2} \\ \frac{1}{2}  & \frac{-1}{m_f+L_f}
\end{bmatrix} + \begin{bmatrix}
0 & 0 \\ 0 & s
\end{bmatrix}\preceq 0 \quad s \geq 0.
\end{align}
This LMI is homogeneous in $(p,\lambda,s)$. We can hence assume $p=1$. Using the Schur Complement, the above LMI is equivalent to
\begin{align}
\begin{bmatrix}
-\rho^2-\lambda \frac{m_fL_f}{m_f+L_f} & \frac{\lambda}{2} & 1 \\ \frac{\lambda}{2} & -\frac{\lambda }{m_f+L_f}+s & -h \\ 1 & -h & -1
\end{bmatrix} \preceq 0.
\end{align}
which is now an LMI in $(\rho^2,\lambda,h,s)$. By treating $s$ as a tuning parameter and minimizing the convergence factor $\rho^2$ over $(\lambda,h)$, we can design stepsizes that yield the best convergence rate for a given level of robustness. Conversely, by treating $\rho^2$ as a tuning parameter and maximizing $s$ over $(\lambda,h)$, we can design stepsizes which yield the largest robustness margin for a desired convergence rate.

\emph{Robust Nesterov's accelerated method.} As our design experiment, we consider the Nesterov's accelerated method applied to a strongly convex $f$:
\begin{align}
x_{k+1} &= y_{k} - h \nabla f(y_k), \\ \nonumber
y_k &= x_k + \beta(x_k-x_{k-1}).
\end{align}
Specifically, we consider the robust version of the LMI in \eqref{eq: LMI exponential}, where the matrices $M_k^i \ i\in \{0,1,2,3\}$ are given in \eqref{eq: Nesterov method LMI matrices} and the robustness matrix is chosen as $s I_{3}, \ s \geq 0$. For a given condition number $\kappa_f = \frac{L_f}{m_f}$ and robustness margin $s$, we use the LMI to compute the convergence factor $\rho$ on the grid $(h,\beta) \in [0 \ \frac{2}{L_f}] \times [0 \ 1]$. See $\S$\ref{sec: Numerical bounds for exponential rates}.

\begin{figure}[h]
	\centering
	\includegraphics[width=0.8\textwidth]{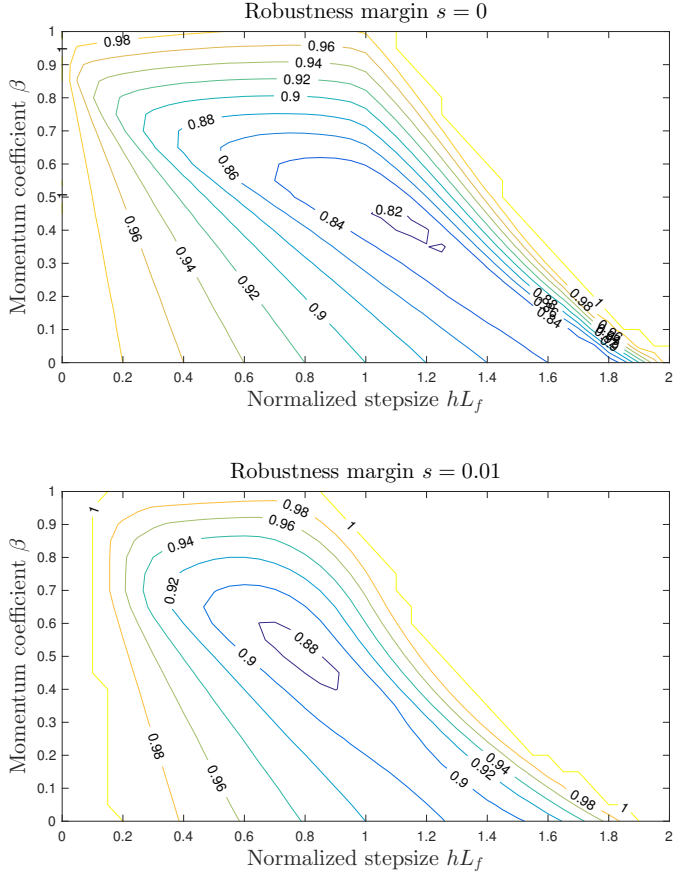}
	\caption{\small Plot of convergence rate $\rho$ of the Nesterov's accelerated method as a function of stepsize $h$ and momentum parameter $\beta$, and for two values of the robustness parameter $s$. Higher values of $s$ increases the robustness of the algorithm at the expense of reduced convergence rate.}
	\label{fig: robust_nesterov_contour}
\end{figure}

In Figure \ref{fig: robust_nesterov_contour}, we plot the contour plots of $\rho$ for $s=0$ and $s=0.01$, respectively. The condition number is fixed at $L_f/m_f=10$. We observe that when $s$ is nonzero, the parameters of the robust algorithm shift towards smaller stepsizes and higher momentum coefficients, leading to higher robustness and lower convergence rates.

\section{Concluding remarks} \label{sec: Conclusions}
In this paper, we have developed an LMI framework, built on the notion of Integral Quadratic Constraints from robust control theory and Lyapunov stability, to certify both exponential and subexponential convergence rates of first-order optimization algorithms. To this end, we proposed a class of time-varying Lyapunov functions that are suitable generating convergence rates in addition to proving stability. We showed that the developed LMI can often be solved in closed form. In particular, we applied the technique to the gradient method, the proximal gradient method, and their accelerated extensions to recover the known analytical upper bounds on their performance. Furthermore, we showed that numerical schemes can also be used to solve the LMI. 

In this paper, we have only used pointwise IQCs to model nonlinearities. More complicated IQCs, such as``off-by-one" IQCs, have shown to be fruitful in improving numerical rate bounds in strongly convex settings \cite{lessard2016analysis}. One direction for future work would be to use these IQCs in tandem with the Lyapunov function used in this paper to further improve the numerical bounds in nonstrongly convex problems. Obtaining better worst-case bounds is useful in a variety of applications, such as Model Predictive Control (MPC). MPC is a sequential optimization-based control scheme, which is particularly useful for constrained and nonlinear control tasks. Implementation of MPC requires the solution of a constrained optimization problem in real time within the sampling period to a specific accuracy determined from stability considerations \cite{richter2012computational}. It is thus important to bound a priori, in a nonconservative manner, the number of iterations needed for a specified accuracy. Improving the numerical rate bounds will allow us to optimize this bound for every problem instance.  More generally, having a nonconservative estimation of convergence rate allows us to compare different algorithms, which must be done by extensive simulations otherwise. We will pursue these applications in future work.



\appendix

\medskip

\section{Symbolic convergence rates for the gradient method} \label{app: Symbolic convergence rates for the gradient method}
The LMI in \eqref{eq: gradient descent weakly convex} with $p=1$ along with the condition $a_{k+1} \geq a_k$ is equivalent to the inequalities
\begin{gather} \label{Symbolic convergence rates for the gradient method}
a_{k+1} \geq a_k, \\
(\dfrac{L_fh^2}{2}-h)a_{k+1}+h^2-\dfrac{\sigma}{L_f} \leq 0, \label{Symbolic convergence rates for the gradient method 1}\\
-\Big(\dfrac{a_{k+1}-a_{k}-2h+\sigma}{2}\Big)^2 \geq 0. \label{Symbolic convergence rates for the gradient method 2}
\end{gather}
The last inequality implies $a_{k+1}=a_{k}+2h-\sigma$. Assuming $a_0=0$ and solving for $a_k$, we obtain $a_k=(2h-\sigma)k$. Therefore, the fastest convergence rate corresponds to the smallest $\sigma$. By substituting $a_k$ in \eqref{Symbolic convergence rates for the gradient method} and \eqref{Symbolic convergence rates for the gradient method 1}, we obtain

\begin{align} \label{app: gradient descent weakly convex}
2h - \sigma \geq 0, \quad (\dfrac{L_fh^2}{2}-h)(2h-\sigma)(k+1)+h^2-\dfrac{\sigma}{L_f} \leq 0.
\end{align}
Since the second inequality must hold for all $k \geq 0$, we must have that ${L_fh^2}/{2}-h \leq 0$ or equivalently, $0 \leq h \leq 2/L_f$. Under this condition, it suffices to ensure the second inequality in \eqref{app: gradient descent weakly convex} holds for $k=0$. This leads to
\begin{align}
\max(0,\dfrac{(L_fh)(L_fh-1)(2h)}{(L_fh)^2-2(L_fh)+2})\leq \sigma \leq 2h.
\end{align}
Therefore, the optimal (minimum) $\sigma$ is
\begin{align}
\sigma_{opt} = \begin{cases}
0 & \mbox{if } 0 \leq hL_f \leq 1 \\
\dfrac{(L_fh)(L_fh-1)(2h)}{(L_fh)^2-2(L_fh)+2} & \mbox{if } 1 < hL_f \leq 2.
\end{cases}
\end{align}
By substituting all the parameters in \eqref{eq: convergence rate discrete}, we obtain
\begin{align}
f(x_k) - f(x_{\star}) \leq \dfrac{\|x_0-x_{\star}\|_2^2}{(2h-\sigma_{opt})k},
\end{align}
which is the same as \eqref{eq: gradient descent weakly convex 1}. \hfill $\square$

\medskip

\section{Proof of Proposition \ref{lemma: IQCs for generalized gradient mapping}}  \label{lemma: IQCs for generalized gradient mapping proof} \emph{Proof of part 1:} Since $g$ is nondifferentiable and convex, it follows from the discussion in $\S$\ref{subsection: IQCs for Proximal Operators} and $\S$\ref{subsection: IQC for projection operators} that $\Pi_{g,h}$ is firmly non-expansive and is hence Lipschitz continuous with Lipschitz parameter equal to one. Further, it is well-known that the map $x \mapsto x-h \nabla f(x)$ is Lipschitz continuous with Lipschitz constant $\gamma_f=\max \{|1-h L_f|,|1- h m_f|\}$; see, for example, \cite{bertsekas2015convex} for a proof. Therefore, the composition $\Pi_{g,h}(x-h\nabla f(x))$ is Lipschitz continuous with parameter $\gamma_f$. In other words, we can write
\begin{align*}
\|\Pi_{g,h}(x-h \nabla f(x))\!-\!\Pi_{g,h}(x_{\star}-h \nabla f(x_{\star}))\|_2^2 \leq \gamma_f^2 \|x-x_{\star}\|_2^2.
\end{align*}
Making the substitution $\Pi_{g,h}(x-h \nabla f(x))=x-h \phi_h(x)$, completing the squares, and rearranging terms yield
\begin{align*} 
\begin{bmatrix}
x-x_{\star} \\ \phi_h(x)-\phi_h(x_{\star})
\end{bmatrix}^\top  \!
\begin{bmatrix}
\dfrac{1}{2h}(\gamma_f^2-1)I_d& \dfrac{1}{2}I_d \\ \dfrac{1}{2}I_d & -\dfrac{h}{2}I_d
\end{bmatrix}		
\begin{bmatrix}
x-x_{\star} \\ \phi_h(x)-\phi_h(x_{\star})
\end{bmatrix} \geq 0.
\end{align*}

\medskip

\emph{Proof of part 2:} First, note that the optimality condition of the proximal operator, defined in \eqref{eq: proximal operator 1}, is that
\begin{align*} 
0 \in \partial g(\Pi_{g,h}(w)) + \dfrac{1}{h} (\Pi_{g,h}(w)-w),
\end{align*}
or equivalently, 
\begin{align} \label{eq: proximal operators optimality condition 1}
0 =T_g(\Pi_{g,h}(w)) + \dfrac{1}{h} (\Pi_{g,h}(w)-w), \ T_g \in \partial g,
\end{align}
where $T_g(w)$ denotes a subgradient of $g$ at $w$. On the other hand, by the definition of the generalized gradient mapping in \eqref{eq: generalized gradient 0}, we have that
\begin{align} \label{eq: generalized gradient explicit 0}
\Pi_{g,h}(y - h \nabla f(y))=y-h  \phi_h(y).
\end{align}
Substituting \eqref{eq: generalized gradient explicit 0} and $w=y-h \nabla f(y)$ in \eqref{eq: proximal operators optimality condition 1}, we can equivalently write $ \phi_h(y)$ as
\begin{align} \label{eq: generalized gradient explicit}
\phi_h(y) = \nabla f(y) + T_g(y-h  \phi_h(y)).
\end{align}
Consider the points $x,y,z  \in  \dom \, f $. We can write
\begin{align*}
f(z) - f(y) &\leq \nabla f(y)^\top (z\!-\!y) \!+\!\dfrac{L_f}{2} \|z\!-\!y\|_2^2, \\
f(y) - f(x) &\leq \nabla f(y)^\top (y\!-\!x)\! -\! \dfrac{m_f}{2}\|y\!-\!x\|_2^2.
\end{align*}
In the first and second inequality, we have used Lipschitz continuity and strong convexity, respectively. Adding both sides yields
\begin{align}  \label{eq: generalized gradient 3}
f(z)\!-\!f(x)\! \leq \!\nabla f(y)^\top (z\!-\!x)+\dfrac{L_f}{2} \|z\!-\!y\|_2^2 \!-\! \dfrac{m_f}{2}\|y\!-\!x\|_2^2.
\end{align}
Further, since $g$ is convex, we can write 
\begin{align}  \label{eq: generalized gradient 4}
g(z) - g(x) \leq T_g(z)^\top (z-x), \quad T_g(z) \in \partial g(z), \quad x,z \in \dom \, g.
\end{align}
Adding both sides of \eqref{eq: generalized gradient 3} and \eqref{eq: generalized gradient 4} for all $x,z \in \dom \, f \cap \dom \, g, \ y \in \dom \, f$, and making the substitutions $z=y-h  \phi_h(y)$ and \eqref{eq: generalized gradient explicit} yields \eqref{eq: composite function inequality}. 

\medskip

\emph{Proof of part 3:} Suppose $\phi_h(y)=0$ for some $y \in \dom \, \phi_h$. It then follows from \eqref{eq: generalized gradient explicit} that $0=\nabla f(y)+T_g(y)$, or equivalently, $0 \in \nabla f(y) + \partial g(y)$. This implies that $y \in \mathcal{X}_{\star}$, according to \eqref{eq: first-order optimality}. Conversely, suppose $y \in \mathcal{X}_{\star}$. We therefore have $\nabla f(y) = -T_g(y)$. Substituting this in \eqref{eq: generalized gradient explicit} yields $\phi_h(y)=T_g(y-h\phi_h(y))-T_g(y)$. Since $T_g$ is monotone, we can write
$$
0 \leq (T_g(y-h\phi_h(y))-T_g(y))^\top (y-h\phi_h(y)-y) = -h \|\phi_h(y)\|_2^2 \quad \mbox{for all } h.
$$
Therefore, we must have that $\phi_h(y)=0$. The proof is now complete. \hfill $\square$

\medskip

\medskip

\section{Proof of Lemma \ref{prop: Quadratic bounds for composite convex functions}} \label{prop: Quadratic bounds for composite convex functions proof}
In order to bound $F(x_{k+1})-F(x_k)$ and $F(x_{k+1})-F(x_{\star})$, we use the inequality
\begin{align} \label{eq: composite function inequality 1}
F(y\!-\!h \phi_h(y))\!-\!F(x) \leq & \phi_h(y)^\top (y\!-\!x)\!-\! \dfrac{m_f}{2} \|y\!-\!x\|_2^2+(\frac{1}{2}L_f h^2\!-\!h) \| \phi_h(y)\|_2^2,
\end{align}
which we proved in Proposition \ref{lemma: IQCs for generalized gradient mapping}. Specifically, we substitute $(x,y)=(x_{\star},y_k)$ in \eqref{eq: composite function inequality 1} to get
\begin{align*} 
&F(x_{k+1})\!-\!F(x_{\star}) \leq (u_k-u_{\star})^\top (y_k\!-\!y_{\star}) \!+\! (\dfrac{L_fh^2}{2}\!-\!h) \|u_k-u_{\star}\|_2^2 \! - \! \dfrac{m_f}{2} \|y_k\!-\!y_{\star}\|_2^2 \\
& = \begin{bmatrix}
y_k-y_{\star} \\ u_k-u_{\star}
\end{bmatrix}^\top \begin{bmatrix}
-\frac{m_f}{2} & \frac{1}{2}\\ \frac{1}{2}& (\frac{1}{2}L_fh^2\!-\!h)
\end{bmatrix}\begin{bmatrix}
y_k-y_{\star} \\ u_k-u_{\star}
\end{bmatrix}. \nonumber \\
& = \begin{bmatrix}
\xi_{k} -\xi_{\star} \\ u_k - u_{\star}
\end{bmatrix}^\top \begin{bmatrix}
C_k & 0 \\ 0 & I_d
\end{bmatrix}^\top \begin{bmatrix}
-\frac{m_f}{2} & \frac{1}{2}\\ \frac{1}{2}& (\frac{1}{2}L_fh^2\!-\!h)
\end{bmatrix}  \begin{bmatrix}
C_k & 0 \\ 0 & I_d
\end{bmatrix} \begin{bmatrix}
\xi_{k} -\xi_{\star} \\ u_k - u_{\star}
\end{bmatrix} \nonumber \\
&=e_k^\top M_k^2 e_k. \nonumber
\end{align*}
where we have used the identities $u_{\star} = \phi_{h}(y_{\star})=0$ and $y_k-y_\star=C_k(\xi_k-\xi_{\star})$. Similarly, in \eqref{eq: composite function inequality 1} we substitute $(x,y)=(x_k,y_k)$ to obtain
\begin{align} \label{thm: proximal 4} 
&F(x_{k+1})\!-\!F(x_k) \leq (u_k-u_{\star})^\top(y_k-x_k) \!+\! (\frac{1}{2}L_fh^2\!-\!h) \|u_k\!-\!u_{\star}\|_2^2\!-\!\dfrac{m_f}{2} \|y_k\!-\!x_k\|_2^2  \\
& = \begin{bmatrix}
y_k-x_k \\ u_k-u_{\star}
\end{bmatrix}^\top \begin{bmatrix}
-\frac{m_f}{2} & \frac{1}{2}\\ \frac{1}{2}& (\frac{1}{2}L_fh^2\!-\!h)
\end{bmatrix}\begin{bmatrix}
y_k-x_k \\ u_k-u_{\star}
\end{bmatrix} \nonumber \\
& = \begin{bmatrix}
\xi_{k} -\xi_{\star} \\ u_k - u_{\star}
\end{bmatrix}^\top \begin{bmatrix}
C_k\!-\!E_k & 0 \\ 0 & I_d
\end{bmatrix}^\top \begin{bmatrix}
-\frac{m_f}{2} & \frac{1}{2}\\ \frac{1}{2}& (\frac{1}{2}L_fh^2\!-\!h)
\end{bmatrix}  \begin{bmatrix}
C_k\!-\!E_k & 0 \\ 0 & I_d
\end{bmatrix} \begin{bmatrix}
\xi_{k} -\xi_{\star} \\ u_k - u_{\star}
\end{bmatrix} \nonumber \\
&=e_k^\top M_k^1 e_k. \nonumber
\end{align}
where we have used $x_\star=y_\star$ and $y_k-x_k=(C_k-E_k)(\xi_k-\xi_{\star})$ to obtain the second equality. Finally, by Proposition \ref{lemma: IQCs for generalized gradient mapping} $u_k = \phi_h(y_k)$ satisfies the pointwise IQC defined by $(Q_{\phi_h},x_{\star},\phi_h(x_{\star}))$. Therefore, we can write
\begin{align} \label{thm: proximal 6} 
e_k^\top M_k^3 e_k &= \begin{bmatrix}
\xi_{k} -\xi_{\star} \\ u_k - u_{\star}
\end{bmatrix}^\top \begin{bmatrix}
C_k & 0 \\ 0 & I_d
\end{bmatrix}^\top Q_{\phi_h} \begin{bmatrix}
C_k & 0 \\ 0 & I_d
\end{bmatrix} \begin{bmatrix}
\xi_{k} -\xi_{\star} \\ u_k - u_{\star}
\end{bmatrix} \\ \nonumber
&= \begin{bmatrix}
y_k - y_{\star} \\ u_k-u_{\star}
\end{bmatrix}^\top Q_{\phi_h}\begin{bmatrix}
y_k - y_{\star} \\ u_k-u_{\star}
\end{bmatrix} \\ \nonumber &\geq 0,
\end{align}
where we have used the identity $y_k-y_{\star}=C_k(\xi_{k}-\xi_{\star})$ to obtain the second inequality. The proof is complete.

\bibliographystyle{siamplain}
\bibliography{Refs}
\end{document}